\documentclass[11pt]{article}
\usepackage[margin=1in]{geometry}
\usepackage{amsmath,amssymb,amsthm}
\usepackage{booktabs}
\usepackage{algorithm,algorithmic}
\usepackage{microtype}
\usepackage{xcolor}
\usepackage{tikz}
\usetikzlibrary{calc,decorations.markings}
\definecolor{cbB}{RGB}{0,114,178}
\definecolor{cbT}{RGB}{230,159,0}
\definecolor{cbU}{RGB}{86,180,233}
\definecolor{cbC}{RGB}{0,158,115}
\definecolor{cbAcc}{RGB}{213,94,0}
\definecolor{cbPur}{RGB}{204,121,167}
\tikzset{
  figlbl/.style={font=\small},
  figsmall/.style={font=\fontsize{9}{11}\selectfont, text=black!85},
  figaxis/.style={black!35, line width=0.3pt},
  figtick/.style={black!45, line width=0.3pt},
  figticklbl/.style={font=\fontsize{9}{11}\selectfont, text=black!75},
  figseg/.style={line width=1.1pt},
  figdim/.style={stealth-stealth, black!55, line width=0.35pt},
  figlead/.style={black!55, line width=0.3pt},
}

\usepackage[font=small]{caption}
\usepackage[colorlinks=true,allcolors=blue,pdfusetitle]{hyperref}
\usepackage[capitalise]{cleveref}
\crefname{figure}{Figure}{Figures}

\newcommand{\AuthorList}{Ethan Keller}
\newcommand{\GithubURL}{\url{https://github.com/ethan-keller/moser-worm-improved-bounds}}
\hypersetup{pdfauthor={Ethan Keller}}

\makeatletter
\def\thm@space@setup{%
  \thm@preskip=16pt plus 4pt minus 2pt
  \thm@postskip=16pt plus 4pt minus 2pt}
\makeatother
\newtheorem{theorem}{Theorem}[section]

\newtheorem{lemma}[theorem]{Lemma}

\theoremstyle{definition}
\newtheorem{definition}[theorem]{Definition}
\theoremstyle{remark}

\numberwithin{equation}{section}

\newcommand{\R}{\mathbb{R}}
\newcommand{\conv}{\operatorname{conv}}
\newcommand{\area}{\operatorname{area}}
\newcommand{\ellg}{\ell}
\newcommand{\tr}{\operatorname{tr}}
\newcommand{\aopt}{\alpha}
\newcommand{\lboundvalue}{0.239}
\newcommand{\uboundvalue}{\frac{7170601298323360123534337}{29109471743520000000000000}}
\newcommand{\ubounddecimal}{0.2463322372\ldots}
\newcommand{\uboundshort}{0.24633\ldots}
\newcommand{\Kub}{K}
\newcommand{\Sfam}{S^{*}}
\newcommand{\lbsym}{\alpha_{\text{lower}}}
\newcommand{\lean}[1]{\texttt{#1}}

\AtBeginDocument{%
  \setlength{\abovedisplayskip}{6pt plus 2pt minus 3pt}%
  \setlength{\belowdisplayskip}{6pt plus 2pt minus 3pt}%
  \setlength{\abovedisplayshortskip}{2pt plus 1pt}%
  \setlength{\belowdisplayshortskip}{4pt plus 1pt}%
}

\title{Improved bounds for universal convex covers of unit arcs}
\author{\AuthorList}
\date{}

\begin{document}
\maketitle

\begin{abstract}
\noindent
Moser's worm problem asks for a planar region of least area containing a congruent copy of every unit arc.
We show that the infimum area $\aopt$ among convex universal covers satisfies $\lboundvalue\le\aopt\le\uboundshort$, reducing the gap between the previous refereed bounds by over $75\%$.
For the lower bound, we choose four unit polygonal arcs and prove by finite subdivision that, however they are placed, their convex hull has area at least $\lboundvalue$.
For the upper bound, we construct a quadrilateral of area $\uboundshort$ and prove cover universality by showing that its support inequalities force uncovered arcs to have length greater than one.
The full proof is formalized in Lean 4 and verified by the Lean kernel. 
Code and certificates are available at \GithubURL.
\end{abstract}

\section{Introduction}
\label{sec:intro}

In the 1960s, Leo Moser asked for the minimum area of a planar region that contains a congruent
copy of every arc of length one (a ``worm'') \cite{moser1966}. This question, known as Moser's worm problem, 
is a classical universal covering problem in discrete geometry. Despite its simple appearance, the problem remains open to this day.

Norwood, Poole and Laidacker proved that a convex universal cover of minimum area
exists \cite{npl1992,kps2013}. In this paper we study this convex case. The optimal
region remains unknown, so work on the problem has focused on bounds for
its area.

\paragraph{Lower bound.}
Wetzel \cite{wetzel1973} proved a lower bound $0.21946$ by combining the diameter
forced by a unit segment with the minimum width forced by Schaer's
broadworm \cite{schaer1968}.
Khandhawit and Sriswasdi \cite{ks2007}
improved this to $0.227498$ using analytic hull-area estimates for three
unit arcs with segment, triangular, and square hulls.
Khandhawit, Pagonakis and Sriswasdi \cite{kps2013} obtained $0.232239$
by combining hull-height estimates for segment, triangular, and
rectangular hulls with the broadworm's minimum-width constraint.
An unrefereed work by Mazur \cite{mazur2026} reports a lower bound
of $0.23743658$ using mixed-area inequalities.

\paragraph{Upper bound.}
Gerriets and Poole \cite{gp1974} constructed a trimmed rhombus of area
below $0.2861$. Norwood, Poole and Laidacker \cite{npl1992} reduced the
bound to $0.27524$ using a circular sector and two attached triangles,
and Wang \cite{wang2006} subsequently obtained $0.270912$.
Panraksa and Wichiramala \cite{pw2021} used support-line contacts and
reflection arguments to prove that a circular sector of radius one and
angle $\pi/6$ covers every unit arc, giving the bound
$\pi/12=0.261799\ldots$.
These covers are convex. The nonconvex constructions of Norwood and
Poole \cite{np2003} and Ploymaklam and Wichiramala \cite{plo2018} concern
the unrestricted problem. Two unrefereed preprints obtain triangular covers by bounding
the length of arcs that cannot fit into a triangle and then scaling it.
Wichiramala and Panraksa \cite{ww-preprint} use interval-certified
support-contact inequalities to obtain a scaled
$30^\circ$--$60^\circ$--$90^\circ$ triangle of area $0.260956\ldots$.
Deng \cite{deng2026} uses balanced support inequalities and exact
rational certificates for an isosceles triangle, obtaining
$0.257883595\ldots$.

In this work we prove \cref{thm:main}, giving a new lower bound of
$\lboundvalue$ and an upper bound of $\uboundshort$. The theorem states
the exact rational upper bound.
The gap between the previous refereed bounds $0.232239$ and $\pi/12$
\cite{kps2013,pw2021} is reduced by approximately $75.2\%$.
Relative to the recent unrefereed bounds reported by Mazur
\cite{mazur2026} and Deng \cite{deng2026}, the reduction is
approximately $64.1\%$. See \cref{tab:bounds-summary}.

\begin{table}[!htbp]
  \centering
  \small
  \begin{tabular}{lccc}
    \toprule
    Bounds & Lower bound & Upper bound & References\\
    \midrule
    Previous refereed
      & $0.232239$ & $\pi/12=0.261799\ldots$
      & \cite{kps2013,pw2021}\\
    Recent unrefereed preprints
      & $0.23743658$ & $0.257883595\ldots$
      & \cite{mazur2026,deng2026}\\
    This paper
      & $\lboundvalue$ & $\ubounddecimal$
      & \cref{thm:main}\\
    \bottomrule
  \end{tabular}
  \caption{Comparison of bounds for convex universal covers.
  Decimal values with ellipses are approximate.}
  \label{tab:bounds-summary}
\end{table}

\begin{theorem}\label{thm:main}
The minimum area $\aopt$ of a convex universal cover of unit arcs satisfies
\[
  \lboundvalue\le\aopt\le\uboundvalue=\ubounddecimal.
\]
\end{theorem}

The lower-bound proof follows the finite-family approach of \cite{kps2013}.
We choose four unit polygonal arcs, reduce their joint placements
to a bounded range of translations and rotations, and divide them into
boxes. In each leaf box, a shoelace expression certifies that
the joint hull has area at least $\lboundvalue$.

For the upper bound, we construct a quadrilateral and prove that it is a universal cover.
We first give the equations for the sides of a quadrilateral and then show that any unit arc it fails to cover
forces a finite system of inequalities that is not compatible with length one. 

Lower and upper arguments are both computer-assisted, fully formalized in Lean 4 and verified by the Lean kernel.

After the definitions in \cref{sec:prelim}, we give the lower-bound
argument in \cref{sec:lb} and the upper-bound argument in \cref{sec:ub}.
\Cref{sec:verification} describes the certificates and formalization.
\Cref{sec:conclusion} concludes and discusses further directions.

\section{Preliminaries}\label{sec:prelim}

We work in the Euclidean plane $\R^2$, with dot product $u\cdot v$
and norm $|u|=\sqrt{u\cdot u}$. For $S\subseteq\R^2$, we write
$\conv(S)$ for its convex hull. When $S$ is measurable, $\area(S)$
denotes its planar Lebesgue area, which may be infinite.

\begin{definition}[Arcs]
An \emph{arc} (or \emph{curve}) is a continuous map
$\gamma:[0,1]\to\R^2$ of finite length, where
\[
  \ellg(\gamma)=
  \sup_{0=t_0<\cdots<t_k=1}
    \sum_{j=0}^{k-1}|\gamma(t_{j+1})-\gamma(t_j)|.
\]
The supremum is over all finite partitions of $[0,1]$.
Self-intersections and retracing are allowed.
A \emph{unit arc} has length one. The \emph{trace} of $\gamma$ is the
set of points it visits.
\[
  \tr(\gamma)=\{\gamma(t):t\in[0,1]\}.
\]
\end{definition}

\emph{Rigid motions} are distance-preserving maps of the plane.
They form the Euclidean group $E(2)$ and have the form
\[
  g(x)=Mx+b,\qquad M\in O(2),\quad b\in\R^2,
\]
where $O(2)$ is the group of real orthogonal $2\times2$ matrices.
These motions include translations, rotations and reflections.
Two sets are \emph{congruent} if a rigid motion maps one onto the other.

\begin{definition}[Covering]
\label{def:covering}
A set $D\subseteq\R^2$ \emph{covers} a set $S\subseteq\R^2$ if it
contains a congruent copy of $S$, that is, $g(S)\subseteq D$ for some
$g\in E(2)$. It covers an arc $\gamma$ if it covers $\tr(\gamma)$.
A \emph{universal cover} covers every unit arc. A \emph{convex universal
cover} is a universal cover that is convex.
\end{definition}

We study the infimum area
\[
  \aopt=\inf\{\area(D):D\subseteq\R^2
    \text{ is a convex universal cover}\}.
\]

\section{The lower bound}\label{sec:lb}

A convex universal cover must contain a congruent copy of each arc
in any finite family of unit arcs, with a separate rigid motion allowed for each
(\cref{def:covering}). We choose the four polygonal arcs of
\cref{sec:worm-family} and prove that every joint convex hull has area
at least $\lbsym=\lboundvalue$.
We first reduce the placement domain (\cref{sec:lb-domain}), then
bound the area on boxes of placements (\cref{sec:lb-fans,sec:lb-box}).
A finite subdivision certifies the bound over the whole reduced domain
(\cref{sec:lb-certificate}).

\subsection{The worm family}
\label{sec:worm-family}

We construct a set of four polygonal unit arcs (the ``worm family'').

\begin{itemize}
  \item The $L$ worm is the line segment with endpoints $(-1/2,0)$ and $(1/2,0)$.
  \item The $T$ worm is the two sides of an equilateral triangle formed by joining two segments of length $1/2$ at a $\pi/3$ angle. Successive vertices are positioned at $(-1/4,0), (0,\sqrt3/4), (1/4,0)$.
  \item The $U$ worm forms a splayed U-shape with three sides of length $1/3$ and two interior angles of $13\pi/24$.
  With $\beta=11\pi/24$, the four successive vertices are positioned at
\[
  z_0 = (0, 0), \qquad z_j=\frac13\sum_{r=0}^{j-1}\bigl(\cos(r\beta),\sin(r\beta)\bigr),
  \qquad j\in\{1,2,3\}.
\]

  \item The $C$ worm follows a crinkled pattern. With $t=\frac{\sqrt{8745}}{1024}$, the successive vertices are at
  \[
    (0,0),\qquad(7/128,t),\qquad(105/128,-t),\qquad(7/8,0).
  \]
Its side lengths are $109/1024$, $806/1024$, $109/1024$. 
\end{itemize}

All four arcs have length one. Set $\Sfam=\{L,T,U,C\}$. Their convex hulls are a segment, an equilateral
triangle, an isosceles trapezoid, and a parallelogram
(\cref{fig:family}). \Cref{fig:joint-placement} shows the placement of this family with the smallest
hull area found numerically $0.23943903$.

\begin{theorem}\label{thm:lb}
Every convex set covering all four arcs in $\Sfam$ has area at least $\lbsym=\lboundvalue$. Consequently $\aopt\ge\lboundvalue$.
\end{theorem}

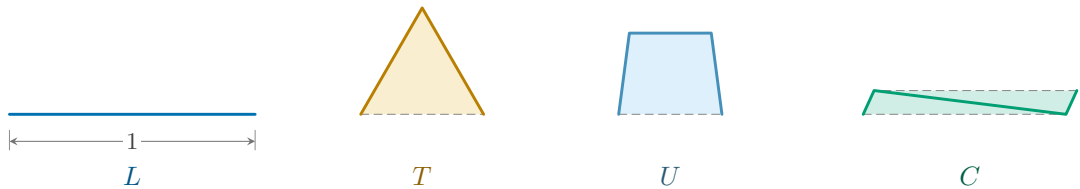
\begin{figure}[hbp]
  \centering
  \begingroup
\pgfmathsetmacro{\veeh}{sqrt(3)/4}
\pgfmathsetmacro{\ubase}{1/6}
\pgfmathsetmacro{\uouter}{(1+2*cos(82.5))/6}
\pgfmathsetmacro{\uheight}{sin(82.5)/3}
\pgfmathsetmacro{\crowny}{sqrt(8745)/1024}
\pgfmathsetmacro{\crownangle}{atan(128*\crowny/105)}
\pgfmathsetmacro{\crownshift}{(7/16)*sin(\crownangle)}
\begin{tikzpicture}[x=3.25cm,y=3.25cm,line join=round,line cap=round]
  \begin{scope}
    \draw[figseg,cbB] (-.5,0)--(.5,0);
    \draw[figtick] (-.5,-.065)--(-.5,-.145);
    \draw[figtick] (.5,-.065)--(.5,-.145);
    \draw[figdim] (-.5,-.11)--(.5,-.11);
    \node[figticklbl,fill=white,inner sep=1pt] at (0,-.11) {$1$};
    \node[figlbl,text=cbB!75!black] at (0,-.25) {$L$};
  \end{scope}
  \begin{scope}[shift={(1.18,0)}]
    \path[fill=cbT!18]
      (-.25,0)--(.25,0)--(0,\veeh)--cycle;
    \draw[black!45,densely dashed,line width=.4pt] (-.25,0)--(.25,0);
    \draw[figseg,cbT!80!black] (-.25,0)--(0,\veeh)--(.25,0);
    \node[figlbl,text=cbT!60!black] at (0,-.25) {$T$};
  \end{scope}
  \begin{scope}[shift={(2.19,0)}]
    \path[fill=cbU!20]
      (-\uouter,0)--(\uouter,0)--(\ubase,\uheight)
      --(-\ubase,\uheight)--cycle;
    \draw[black!45,densely dashed,line width=.4pt]
      (-\uouter,0)--(\uouter,0);
    \draw[figseg,cbU!80!black]
      (-\uouter,0)--(-\ubase,\uheight)
      --(\ubase,\uheight)--(\uouter,0);
    \node[figlbl,text=cbU!50!black] at (0,-.25) {$U$};
  \end{scope}
  \begin{scope}[shift={(3.41,\crownshift)},rotate=\crownangle]
    \path[fill=cbC!18]
      (-7/16,0)--(49/128,-\crowny)--(7/16,0)
      --(-49/128,\crowny)--cycle;
    \draw[black!45,densely dashed,line width=.4pt]
      (-7/16,0)--(49/128,-\crowny);
    \draw[black!45,densely dashed,line width=.4pt]
      (7/16,0)--(-49/128,\crowny);
    \draw[figseg,cbC]
      (-7/16,0)--(-49/128,\crowny)
      --(49/128,-\crowny)--(7/16,0);
  \end{scope}
  \node[figlbl,text=cbC!60!black] at (3.41,-.25) {$C$};
\end{tikzpicture}
\endgroup
  \caption{The four unit arcs, at the same scale. Heavy strokes
  are the arcs, and dashed segments complete their shaded convex hulls.}
  \label{fig:family}
\end{figure}

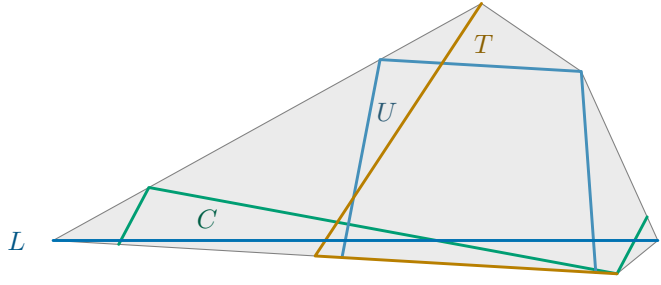
\begin{figure}[htbp]
  \centering
  \begin{tikzpicture}[x=8cm,y=8cm,line join=round,line cap=round]
  \coordinate (L0) at (-0.500000000000,0.000000000000);
  \coordinate (L1) at (0.500000000000,0.000000000000);
  \coordinate (T0) at (0.432457379679,-0.055022582177);
  \coordinate (T1) at (-0.066674396651,-0.025569740923);
  \coordinate (T2) at (0.208398400254,0.391964636587);
  \coordinate (U0) at (0.397380461133,-0.052952758210);
  \coordinate (U1) at (0.373414526423,0.279517909796);
  \coordinate (U2) at (0.040660008870,0.299153137300);
  \coordinate (U3) at (-0.022240416660,-0.028191707748);
  \coordinate (C0) at (-0.391489594092,-0.006402997988);
  \coordinate (C1) at (-0.341616151666,0.087635527595);
  \coordinate (C2) at (0.432457379679,-0.055022582177);
  \coordinate (C3) at (0.482330822105,0.039015943405);

  \path[fill=black!8,draw=black!50,line width=.4pt]
    (L0)--(C0)--(U3)--(U0)--(C2)--(L1)--(U1)--(T2)--(U2)--cycle;

  \draw[figseg,cbC] (C0)--(C1)--(C2)--(C3);
  \draw[figseg,cbU!80!black] (U0)--(U1)--(U2)--(U3);
  \draw[figseg,cbT!80!black] (T0)--(T1)--(T2);
  \draw[figseg,cbB] (L0)--(L1);

  \node[figlbl,text=cbB!75!black,anchor=east,inner sep=1pt] at (-.54,0) {$L$};
  \node[figlbl,text=cbT!60!black] at (.212,.325) {$T$};
  \node[figlbl,text=cbU!50!black] at (.052,.212) {$U$};
  \node[figlbl,text=cbC!60!black] at (-.245,.035) {$C$};

\end{tikzpicture}
  \caption{The four unit arcs in their best recorded numerical joint
  placement, with their convex hull shaded and area approximately $0.23943903$.}
  \label{fig:joint-placement}
\end{figure}

\subsection{Bounding the placement domain}
\label{sec:lb-domain}

A \emph{placement} assigns an independent rigid motion to each of the
four arcs. Let $H$ be their \emph{joint hull}, the convex hull of the
placed traces. We show that every placement with $\area(H)<\lbsym$
can be represented in a bounded parameter box without changing
$\area(H)$.

Replacing each arc by its convex hull leaves $H$ unchanged.
A global rigid motion preserves $\area(H)$, so we can fix $L$.

For $i\in\{T,U,C\}$, translate its convex hull so that the average 
of its vertices is the origin. The resulting vertices $v_{i,j}$ are
indexed counterclockwise from the leftmost vertex. 
Here $0\le j<m_i$, with $m_T=3$ and $m_U=m_C=4$.
Write $t_i=(t_{ix},t_{iy})$ for translation and $R(\phi_i)$ for
rotation through angle $\phi_i$. We allow reflection in the $x$-axis
before rotation.

\paragraph{Full placement space.}
With $L$ fixed, let $\varepsilon_i=1$ if reflection in the $x$-axis is
applied to hull $i$, and $\varepsilon_i=0$ otherwise. Allowing the
angles and translations to be arbitrary real numbers gives the
parameterization
\[
  (\phi_T,\phi_U,\phi_C,\,
   t_{Tx},t_{Ty},t_{Ux},t_{Uy},t_{Cx},t_{Cy};\,
   \varepsilon_T,\varepsilon_U,\varepsilon_C)
  \in\R^9\times\{0,1\}^3.
\]

\paragraph{Eliminating reflections.}
The hulls of $T$ and $U$ have reflection symmetries, so every reflected
copy of either hull can also be obtained by rotation and translation.
If the hull of $C$ is reflected, reflect the entire configuration in
the $x$-axis. This fixes $L$, preserves $\area(H)$, and removes the reflection
choice for $C$. The resulting copies of $T$ and $U$ can again be
represented without reflections. We may therefore take
$\varepsilon_T=\varepsilon_U=\varepsilon_C=0$. The remaining parameter
domain is $\R^9$, and each placed vertex has the form
$t_i+R(\phi_i)v_{i,j}$.

\paragraph{Bounding angles.}
First reduce all three angles modulo $2\pi$.
The hulls of $T$ and $C$ have rotational periods $2\pi/3$ and $\pi$,
so their angles may further be restricted to $[0,2\pi/3]$ and $[0,\pi]$.
A global half-turn fixes $L$ and preserves $\area(H)$. It adds $\pi$ to
$\phi_T$, which is equivalent to adding $\pi/3$ modulo $2\pi/3$.
Choosing whether to apply this half-turn therefore halves the remaining
range for $\phi_T$. We obtain
\begin{equation}\label{eq:angles}
  (\phi_T,\phi_U,\phi_C)\in[0,\pi/3]\times[0,2\pi]\times[0,\pi].
\end{equation}

\paragraph{Bounding translations.}
Each translation still ranges over $\R^2$. We show that $\area(H)<\lbsym$
forces all three hull centers into a fixed rectangle. The following
area estimate provides the needed bound.
The \emph{width} (or \emph{extent}) of a nonempty compact set $S$ in
a unit direction $u$ is
$\max_{x\in S}u\cdot x-\min_{x\in S}u\cdot x$.

\begin{lemma}[chord bound]\label{lem:chord}
Let $H\subseteq\R^2$ be convex and contain a segment of length $d>0$
and a nonempty compact set $S$. If $e$ is the extent of $S$
perpendicular to the segment, then
\[
  \area(H)\ge\frac{de}{2}.
\]
\end{lemma}

\begin{proof}
Let $S'$ be the union of $S$ and the set of the segment's endpoints. Let $h_+$ and
$h_-$ be the maximum perpendicular distances attained by $S'$ in the
two closed half-planes bounded by the segment's line, respectively.
Thus $h_+,h_-\ge0$ and $h_++h_-\ge e$.
The segment and the corresponding extreme points form two possibly
degenerate triangles $\Delta_+,\Delta_-\subseteq H$
(\cref{fig:chord-bound}). Their interiors are disjoint, so
\[
  \area(H)\ge\area(\Delta_+)+\area(\Delta_-)
  =\frac d2(h_++h_-)\ge\frac{de}{2}.
\]
\end{proof}

\begin{figure}[htbp]
  \centering
  \begin{tikzpicture}[line join=round,line cap=round]
  \def\chordHull{(-.3,0)--(.3,-1.1)--(2.8,-1.2)--(4,0)
    --(3.2,1.25)--(1.2,1.4)--cycle}
  \begin{scope}[x=1.1cm,y=1.1cm]
    \node[figlbl] at (1.8,1.85) {(a) $S$ on both sides of the line};
    \path[fill=black!3,draw=black!40,line width=.45pt] \chordHull;
    \path[fill=cbB!15,draw=cbB!70,line width=.5pt]
      (0,0)--(3.6,0)--(1.8,.95)--cycle;
    \path[fill=cbAcc!15,draw=cbAcc!75,line width=.5pt]
      (0,0)--(3.6,0)--(1.8,-.75)--cycle;
    \draw[black!65,densely dashed,line width=.5pt]
      (1.8,.1) ellipse[x radius=.5,y radius=.85];
    \node[figlbl] at (1.8,.43) {$S$};
    \node[figlbl,text=black!60] at (3.35,.55) {$H$};
    \node[figsmall,text=cbB!80!black] at (.9,.24) {$\Delta_+$};
    \node[figsmall,text=cbAcc!80!black] at (.9,-.22) {$\Delta_-$};
    \draw[figaxis,densely dotted] (-.78,0)--(4.48,0);
    \draw[figtick,densely dotted] (-.78,.95)--(4.48,.95);
    \draw[figtick,densely dotted] (-.78,-.75)--(4.48,-.75);
    \draw[figdim] (-.62,0)--(-.62,.95)
      node[midway,left,figlbl] {$h_+$};
    \draw[figdim] (-.62,-.75)--(-.62,0)
      node[midway,left,figlbl] {$h_-$};
    \draw[figdim] (4.32,-.75)--(4.32,.95)
      node[midway,right,figlbl] {$e$};
    \draw[figseg,cbB] (0,0)--(3.6,0);
    \foreach \p in {(0,0),(3.6,0),(1.8,.95),(1.8,-.75)}
      \fill[black!75] \p circle[radius=1.5pt];
    \draw[figtick] (0,-.10)--(0,-1.58);
    \draw[figtick] (3.6,-.10)--(3.6,-1.58);
    \draw[figdim] (0,-1.48)--(3.6,-1.48)
      node[midway,fill=white,inner sep=1pt,figlbl] {$d$};
    \node[figlbl] at (1.8,-1.96) {$h_++h_-=e$};
  \end{scope}
  \begin{scope}[shift={(7.5,0)},x=1.1cm,y=1.1cm]
    \node[figlbl] at (1.8,1.85) {(b) $S$ on one side of the line};
    \path[fill=black!3,draw=black!40,line width=.45pt] \chordHull;
    \path[fill=cbB!15,draw=cbB!70,line width=.5pt]
      (0,0)--(3.6,0)--(1.8,1.15)--cycle;
    \draw[black!65,densely dashed,line width=.5pt]
      (1.8,.75) ellipse[x radius=.5,y radius=.4];
    \node[figlbl] at (1.8,.75) {$S$};
    \node[figlbl,text=black!60] at (3.35,.55) {$H$};
    \node[figsmall,text=cbB!80!black] at (.9,.24) {$\Delta_+$};
    \node[figsmall,text=cbAcc!80!black] at (.9,-.22) {$\Delta_-$};
    \draw[figaxis,densely dotted] (-.78,0)--(4.48,0);
    \draw[figtick,densely dotted] (-.78,1.15)--(4.48,1.15);
    \draw[figtick,densely dotted] (1.8,.35)--(4.48,.35);
    \draw[figdim] (-.62,0)--(-.62,1.15)
      node[midway,left,figlbl] {$h_+$};
    \draw[figdim] (4.32,.35)--(4.32,1.15)
      node[midway,right,figlbl] {$e$};
    \draw[figseg,cbB] (0,0)--(3.6,0);
    \foreach \p in {(0,0),(3.6,0),(1.8,1.15),(1.8,.35)}
      \fill[black!75] \p circle[radius=1.5pt];
    \draw[figtick] (0,-.10)--(0,-1.58);
    \draw[figtick] (3.6,-.10)--(3.6,-1.58);
    \draw[figdim] (0,-1.48)--(3.6,-1.48)
      node[midway,fill=white,inner sep=1pt,figlbl] {$d$};
    \node[figlbl] at (1.8,-1.96) {$h_-=0,\qquad h_+\ge e$};
  \end{scope}
\end{tikzpicture}
  \caption{The same convex set $H$ and base segment of length $d$ in both
  cases of the chord bound. The triangles $\Delta_+$ and $\Delta_-$ have
  heights $h_+$ and $h_-$, whose sum is at least the perpendicular extent
  $e$ of $S$. In (b), $\Delta_-$ degenerates to the base segment.}
  \label{fig:chord-bound}
\end{figure}
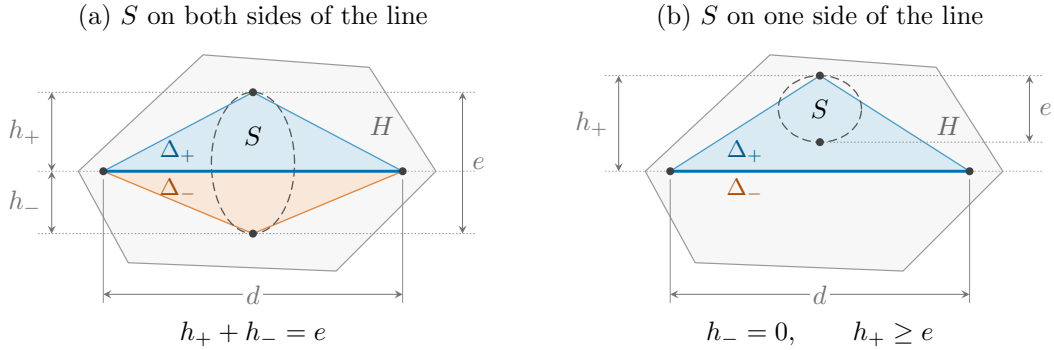

We first bound the horizontal translations $t_{ix}$ using this chord estimate.
We know that the hull of worm $T$ has width at least $w=\sqrt3/4$ in every direction, because
one of its sides must project onto any direction with length at least
$\tfrac12\cos(\pi/6)=w$. We also know that the longest chord in $H$ is
its \emph{diameter} $\operatorname{diam}(H)=\max_{x,y\in H}|x-y|$.
Apply \cref{lem:chord} with a longest chord
of $H$ as the segment and the hull of $T$ as $S$. The perpendicular
extent of $S$ is at least $w$, so $\area(H)\ge w\,\operatorname{diam}(H)/2$,
and hence $\operatorname{diam}(H)\le2\area(H)/w$.

We can now bound the horizontal translations $t_{ix}$. Since the vertices of each hull were centered at the origin,
their average after placement is $t_i$, which lies in $H$ by convexity. The endpoints of $L$ are $e_-=(-1/2,0)$ and $e_+=(1/2,0)$.
The larger of the two horizontal separations from $t_i$ to these endpoints is
$|t_{ix}|+\tfrac12$, which is no greater than the
corresponding Euclidean distance. Since $t_i,e_-,e_+\in H$, both
endpoint distances are at most $\operatorname{diam}(H)$. Thus
\[
  |t_{ix}|+\tfrac12\le\max_\pm|t_i-e_\pm|
  \le\operatorname{diam}(H)\le2\area(H)/w.
\]
To bound the vertical translations $t_{iy}$, consider the triangle
$\conv(\{e_-,e_+,t_i\})$, which also lies in $H$ by convexity.
Its base $L$ has length one and its height is $|t_{iy}|$. Therefore
\[
  \tfrac12|t_{iy}|\le\area(H),
  \qquad\text{hence}\qquad |t_{iy}|\le2\area(H).
\]

Together with the angular reductions, these
bounds show that every placement with $\area(H)<\lbsym$ has a representative
with the same joint hull area in the compact box
\begin{equation}\label{eq:domain}
\begin{aligned}
 \mathcal D&=[0,\pi/3]\times[0,2\pi]\times[0,\pi]
       \times\bigl([-X,X]\times[-Y,Y]\bigr)^3,\\
 X&=\frac{2\lbsym}{w}-\frac12<0.603894,\qquad
 Y=2\lbsym=0.478.
\end{aligned}
\end{equation}

\subsection{Polygon area bounds}\label{sec:lb-fans}

For a fixed placement, we obtain lower bounds for $\area(H)$ from ordered lists of vertices drawn from the four placed hulls.
We distinguish lists with at most five vertices from longer lists.
Lists of at most five vertices admit bounds valid in every order, whereas longer lists require geometric conditions.

Write $u\times v=u_xv_y-u_yv_x$ and $v^\perp=(-v_y,v_x)$.
Let $\mathcal P=(P_0,\ldots,P_{m-1})$ be an ordered list of $m\ge3$
vertices from the four placed hulls, using each hull vertex at most once.
Different vertices may coincide after placement.
Define the signed shoelace expression
\[
  F_{\mathcal P}=\frac12\sum_{i=0}^{m-1}P_i\times P_{i+1},
  \qquad P_m=P_0.
\]
Write $\conv(\mathcal P)$ for the convex hull of these points.
For a simple polygon ordered counterclockwise, $F_{\mathcal P}$ is its
ordinary area. Our lists may instead cross themselves or be degenerate,
and their vertices need not lie on the boundary of $H$. We therefore
need to justify the \emph{shoelace bound} $F_{\mathcal P}\le\area(H)$.

\paragraph{At most five vertices.}
\begin{lemma}[small polygon bound]\label{lem:small-polygon}
Let $3\le m\le5$, and let $\mathcal P=(P_0,\ldots,P_{m-1})$
be a list of points in a convex set $H$, in any order. Then
\[
  F_{\mathcal P}\le\area(H).
\]
\end{lemma}

\begin{proof}
For $m=3$, the shoelace expression is the triangle's area with a sign
determined by the vertex order. The triangle, possibly degenerate,
lies in $H$ by convexity. Hence
\[
  F_{\mathcal P}\le|F_{\mathcal P}|
    =\area(\conv(\mathcal P))\le\area(H).
\]
For $m=4$, expanding the four terms of the shoelace sum gives
\[
  F_{\mathcal P}=\tfrac12(P_2-P_0)\times(P_3-P_1).
\]
This expresses the signed area in terms of the two diagonals.
If $P_0=P_2$, the expression is zero. Otherwise put
$d=|P_2-P_0|$ and let $e$ be the extent of $[P_1,P_3]$ perpendicular
to $[P_0,P_2]$. The cross product has magnitude $de$, and both
segments lie in $H$ by convexity. Thus \cref{lem:chord} gives
\[
  |F_{\mathcal P}|=\frac{de}{2}\le\area(H).
\]
The diagonals need not intersect, so this also covers crossed and
degenerate quadrilaterals. For $m=5$, the proof is given in \cref{app:five-point}.
\end{proof}

Thus the small cases
require no sign or nondegeneracy conditions for the polygon bound.

\paragraph{More than five vertices.}
An analogous bound for arbitrary ordered points fails already at $m=6$.
Tracing a nondegenerate triangle counterclockwise twice gives
twice its hull area. Repeated point locations are allowed,
as noted above. We therefore use geometric conditions for longer lists.
The first is a direct fan condition which ensures that the triangles have
disjoint interiors (\cref{fig:fan-corner}(a)).
This condition holds for every $m \ge 3$, even though \cref{lem:small-polygon}
already covers the cases $m \le 5$.

\begin{lemma}[fan bound]\label{lem:fan}
Let $m\ge3$ and set $\xi_i=P_i-P_0$, the displacement vectors from
$P_0$ to the remaining vertices. If
\begin{equation}\label{eq:fan}
  \xi_1\times\xi_i>0\quad(2\le i\le m-1),\qquad
  \xi_i\times\xi_{i+1}>0\quad(2\le i\le m-2),
\end{equation}
then $0<F_{\mathcal P}\le\area(H)$.
\end{lemma}

\begin{proof}
The first inequalities show that every $\xi_i$ with $i\ge2$ lies in the open
half-plane to the left of $\xi_1$. Their angles from $\xi_1$ lie in
$(0,\pi)$, and the remaining inequalities make these angles strictly
increase. The triangles $\conv(\{P_0,P_i,P_{i+1}\})$ with
$1\le i\le m-2$ therefore have disjoint interiors. Each lies in
$\conv(\mathcal P)\subseteq H$, and their areas sum to
\[
  \frac12\sum_{i=1}^{m-2}\xi_i\times\xi_{i+1}=F_{\mathcal P}.
\]
\end{proof}

A separating chord provides another way to justify the same shoelace bound by
splitting the list into two smaller lists and bounding each one (\cref{fig:chord-cut}).
This can apply when the direct fan conditions (\cref{lem:fan}) fail.

\begin{lemma}[chord decomposition]\label{lem:chord-cut}
For $2\le j\le m-2$, put
\[
  \mathcal P_1=(P_0,\ldots,P_j),\qquad
  \mathcal P_2=(P_j,\ldots,P_{m-1},P_0),\qquad
  H_1=\conv(\mathcal P_1),\qquad
  H_2=\conv(\mathcal P_2).
\]
Suppose $P_0\ne P_j$ as points, the two lists lie in opposite closed
half-planes bounded by the line $P_0P_j$, and
$F_{\mathcal P_i}\le\area(H_i)$ for $i=1,2$. Then
\[
  F_{\mathcal P}=F_{\mathcal P_1}+F_{\mathcal P_2}
  \le\area(\conv(\mathcal P))\le\area(H).
\]
\end{lemma}

\begin{proof}
The added edges $P_jP_0$ and $P_0P_j$ cancel in the shoelace sums.
The intersection $H_1\cap H_2$ lies on the chord line, so has area
zero. Hence $\area(H_1)+\area(H_2)=\area(H_1\cup H_2)
\le\area(\conv(\mathcal P))$.
\end{proof}

Changing the starting vertex while keeping the cyclic order preserves
$F_{\mathcal P}$, so a cut according to \cref{lem:chord-cut} may start at any vertex.
Repeated valid cuts reduce the proof to lists that have at most five
vertices (\cref{lem:small-polygon}) or satisfy the fan conditions (\cref{lem:fan}).
Each bound is applied to the convex hull of that list.

Next, we use these shoelace bounds to certify $\area(H)\ge\lbsym$
throughout a \emph{box} of placements.

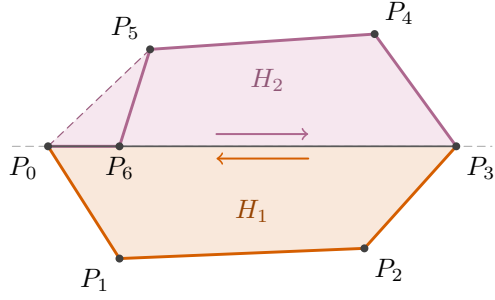
\begin{figure}[!htbp]
  \centering
  \begin{tikzpicture}[x=1.35cm,y=1.35cm,line join=round,line cap=round]
  \coordinate (p0) at (0,0);
  \coordinate (p1) at (.7,-1.1);
  \coordinate (p2) at (3.1,-1);
  \coordinate (p3) at (4,0);
  \coordinate (p4) at (3.2,1.1);
  \coordinate (p5) at (1,.95);
  \coordinate (p6) at (.7,0);
  \path[fill=cbAcc!14] (p0)--(p1)--(p2)--(p3)--cycle;
  \path[fill=cbPur!18] (p0)--(p3)--(p4)--(p5)--cycle;
  \draw[figaxis,densely dashed] (-.35,0)--(4.35,0);
  \draw[figseg,cbAcc] (p0)--(p1)--(p2)--(p3);
  \draw[figseg,cbPur!85!black] (p3)--(p4)--(p5)--(p6)--(p0);
  \draw[cbPur!75!black,densely dashed,line width=.45pt] (p5)--(p0);
  \draw[black!65,line width=.65pt] (p0)--(p3);
  \draw[->,cbAcc,line width=.75pt] (2.55,-.12)--(1.65,-.12);
  \draw[->,cbPur!85!black,line width=.75pt] (1.65,.12)--(2.55,.12);
  \node[figlbl,text=cbAcc!70!black] at (2,-.63) {$H_1$};
  \node[figlbl,text=cbPur!70!black] at (2.15,.65) {$H_2$};
  \foreach \i in {0,...,6}
    \fill[black!75] (p\i) circle[radius=1.5pt];
  \node[figlbl,below left] at (p0) {$P_0$};
  \node[figlbl,below left] at (p1) {$P_1$};
  \node[figlbl,below right] at (p2) {$P_2$};
  \node[figlbl,below right] at (p3) {$P_3$};
  \node[figlbl,above right] at (p4) {$P_4$};
  \node[figlbl,above left] at (p5) {$P_5$};
  \node[figlbl,below] at (p6) {$P_6$};
\end{tikzpicture}
  \caption{A cut with $m=7$ and $j=3$. The child hulls meet only on
  the chord line, and the closing chord edges have opposite orientations.
  The dashed segment $P_0P_5$ completes the boundary of $H_2$ and
  $P_6$ lies on the chord but is not a cut endpoint.}
  \label{fig:chord-cut}
\end{figure}

\subsection{Certifying a placement box}\label{sec:lb-box}

A placement box $\mathcal B$ specifies a closed interval for each of
the three angles and six translation coordinates. We must prove
$\area(H)\ge\lbsym$ for every placement in this box.
The vertex choices and their order in each list remain fixed, while
their positions and shoelace expressions vary. We first choose a
weighted shoelace expression, then use polynomial bounds to verify
the geometric conditions and the numerical threshold throughout
$\mathcal B$.

\paragraph{Combining shoelace bounds.}
One list suffices if $\lbsym\le F_{\mathcal P}\le\area(H)$ throughout
$\mathcal B$. Different shoelace expressions may vary in opposite
directions, so a weighted combination can give a stronger uniform
lower bound. For finitely many lists $\mathcal P_\nu$, choose fixed
rational weights and set
\[
  F=\sum_\nu\lambda_\nu F_{\mathcal P_\nu},\qquad
  \lambda_\nu\ge0,\qquad \sum_\nu\lambda_\nu\le1.
\]
If each bound $F_{\mathcal P_\nu}\le\area(H)$ holds throughout
$\mathcal B$, the weight conditions give
\[
  F\le\Bigl(\sum_\nu\lambda_\nu\Bigr)\area(H)\le\area(H).
\]
It remains to certify the geometric conditions for every list and
prove $F\ge\lbsym$ throughout $\mathcal B$.

\paragraph{Enclosing rotations.}
We replace each angle by cosine and sine coordinates and enclose
the allowed arc by a rectangle (\cref{fig:fan-corner}(b)). These
coordinates are then allowed to vary independently. Together with
the six translation coordinates, this gives twelve variables.

\begin{lemma}[rotation reduction]\label{lem:rotation-reduction}
For $i\in\{T,U,C\}$, suppose the angle interval in $\mathcal B$ is
contained in $[\phi_i^0-\rho_i,\phi_i^0+\rho_i]$, where
$0\le\rho_i<\pi/2$. Write
$c_i=\cos(\phi_i-\phi_i^0)$ and $s_i=\sin(\phi_i-\phi_i^0)$.
Every fixed linear combination $F$ of shoelace expressions agrees,
on actual placements, with a polynomial $\widetilde F$ in the twelve
variables $(t_{ix},t_{iy},c_i,s_i)$ for $i=T,U,C$.
This polynomial has total degree at most two and is affine in each
variable separately (multiaffine). The new coordinates satisfy
\begin{equation}\label{eq:rotation-box}
  c_i\in[\cos\rho_i,1],\qquad
  s_i\in[-\sin\rho_i,\sin\rho_i].
\end{equation}
\end{lemma}

\begin{proof}
Fix a hull $i$. The angle-addition formulas give
\[
  R(\phi_i)v_{i,j}
    =c_iR(\phi_i^0)v_{i,j}
     +s_i\bigl(R(\phi_i^0)v_{i,j}\bigr)^\perp.
\]
Since $\phi_i^0$ is fixed, each placed vertex
$t_i+R(\phi_i)v_{i,j}$ is affine in $t_{ix},t_{iy},c_i,s_i$.
For two vertices $v_{i,j},v_{i,k}$ of the same hull, expansion and
invariance of cross products under rotation give
\[
\begin{aligned}
  &(t_i+R(\phi_i)v_{i,j})\times(t_i+R(\phi_i)v_{i,k})\\
  &\qquad =t_i\times\bigl(R(\phi_i)(v_{i,k}-v_{i,j})\bigr)
            +v_{i,j}\times v_{i,k}.
\end{aligned}
\]
The last term is constant. In the first term, a translation
coordinate is multiplied by at most one of $c_i,s_i$, so no squared
variable remains. We use this identity before treating $c_i,s_i$
as independent variables.

For vertices from different hulls, the two factors involve disjoint
groups of variables. Their cross product therefore also has degree
at most two and is affine in each variable. Edges involving fixed
endpoints are affine. Summing the edge expressions gives
$\widetilde F$.

The cosine and sine bounds in \cref{eq:rotation-box} follow from
$-\rho_i\le\phi_i-\phi_i^0\le\rho_i$.
Cosine is even and decreases on $[0,\rho_i]$, while sine increases
on $[-\rho_i,\rho_i]$ because $\rho_i<\pi/2$.
\end{proof}

A lower bound for $\widetilde F$ on the product of these six intervals
and the six translation intervals is therefore valid for $F$ on
$\mathcal B$.

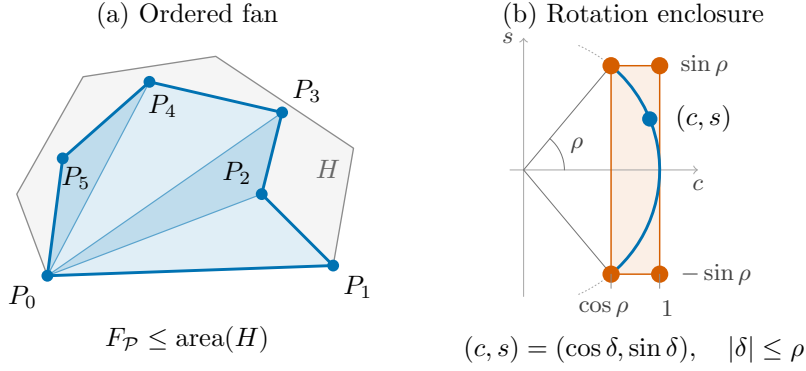
\begin{figure}[t]
  \centering
  \begingroup
\pgfmathsetmacro{\cornerh}{50}
\pgfmathsetmacro{\cornerc}{cos(\cornerh)}
\pgfmathsetmacro{\corners}{sin(\cornerh)}
\begin{tikzpicture}[line join=round,line cap=round]
  \node[figlbl] at (1.85,3.45) {(a) Ordered fan};
  \node[figlbl] at (7.75,3.45) {(b) Rotation enclosure};
  \begin{scope}[scale=1.35]
    \coordinate (p0) at (0,0);
    \coordinate (p1) at (2.8,.1);
    \coordinate (p2) at (2.1,.8);
    \coordinate (p3) at (2.3,1.6);
    \coordinate (p4) at (1,1.9);
    \coordinate (p5) at (.15,1.15);
    \path[fill=black!4,draw=black!45,line width=.5pt]
      (p0)--(p1)--(3,1.25)--(1.65,2.15)--(.35,1.95)--(-.3,.8)--cycle;
    \foreach \i/\j/\tone in {1/2/12,2/3/25,3/4/12,4/5/25}{
      \path[fill=cbB!\tone] (p0)--(p\i)--(p\j)--cycle;
    }
    \foreach \i in {2,3,4}
      \draw[cbB!65,line width=.45pt] (p0)--(p\i);
    \draw[figseg,cbB]
      (p0)--(p1)--(p2)--(p3)--(p4)--(p5)--cycle;
    \foreach \i in {0,...,5}
      \fill[cbB] (p\i) circle[radius=1.65pt];
    \node[figlbl,below left] at (p0) {$P_0$};
    \node[figlbl,below right] at (p1) {$P_1$};
    \node[figlbl,above left] at (p2) {$P_2$};
    \node[figlbl,above right] at (p3) {$P_3$};
    \node[figlbl] at (1.13,1.66) {$P_4$};
    \node[figlbl] at (.28,.96) {$P_5$};
    \node[figlbl,text=black!60] at (2.75,1.05) {$H$};
    \node[figlbl] at (1.37,-.62) {$F_{\mathcal P}\le\operatorname{area}(H)$};
  \end{scope}
  \begin{scope}[shift={(6.3,1.4)},scale=1.8]
    \path[fill=cbAcc!10] (\cornerc,-\corners) rectangle (1,\corners);
    \draw[->,figaxis] (-.10,0)--(1.28,0) node[below,figticklbl] {$c$};
    \draw[->,figaxis] (0,-.93)--(0,.97) node[left,figticklbl] {$s$};
    \draw[figlead] (0,0)--(\cornerc,\corners);
    \draw[figlead] (0,0)--(\cornerc,-\corners);
    \draw[figlead] (.30,0) arc[start angle=0,end angle=\cornerh,radius=.30];
    \node[figticklbl] at (.39,.18) {$\rho$};
    \draw[densely dotted,black!35,line width=.4pt]
      (\cornerc,\corners) arc[start angle=\cornerh,end angle=67,radius=1];
    \draw[densely dotted,black!35,line width=.4pt]
      (\cornerc,-\corners) arc[start angle=-\cornerh,end angle=-67,radius=1];
    \draw[cbAcc,line width=.55pt]
      (\cornerc,-\corners) rectangle (1,\corners);
    \draw[figseg,cbB]
      (\cornerc,-\corners) arc[start angle=-\cornerh,end angle=\cornerh,radius=1];
    \foreach \x in {\cornerc,1}{
      \foreach \y in {-\corners,\corners}
        \fill[cbAcc] (\x,\y) circle[radius=1.8pt];
    }
    \coordinate (rotation) at ({cos(22)},{sin(22)});
    \fill[cbB] (rotation) circle[radius=1.6pt];
    \node[figlbl,anchor=west] at (1.04,.38) {$(c,s)$};
    \draw[figtick] (\cornerc,-\corners)--(\cornerc,-.87);
    \draw[figtick] (1,-\corners)--(1,-.87);
    \node[figticklbl] at ({\cornerc-.055},-1.02) {$\cos\rho$};
    \node[figticklbl] at (1.055,-1.02) {$1$};
    \node[figticklbl,anchor=west] at (1.08,\corners) {$\sin\rho$};
    \node[figticklbl,anchor=west] at (1.08,-\corners) {$-\sin\rho$};
    \node[figlbl] at (.81,-1.32)
      {$(c,s)=(\cos\delta,\sin\delta),\quad |\delta|\le\rho$};
  \end{scope}
\end{tikzpicture}
\endgroup
  \caption{(a) Fan triangles with disjoint interiors lie in the joint hull $H$.
  (b) A rotation arc lies in the rectangle
  $[\cos\rho,1]\times[-\sin\rho,\sin\rho]$ used for the polynomial bound.
  Points of the rectangle need not represent rotations.}
  \label{fig:fan-corner}
\end{figure}

\paragraph{Bounding the polynomial.}
List the coordinates $(t_{ix},t_{iy},c_i,s_i)$ for $i=T,U,C$ as
$x_1,\ldots,x_{12}$, in that order. The translation intervals come
from $\mathcal B$ and the rotation bounds from \cref{eq:rotation-box}.
Enclose each coordinate in $[m_j-r_j,m_j+r_j]$, with rational center
$m_j$ and rational radius $r_j\ge0$, and write
\[
  x_j=m_j+r_ju_j,\qquad |u_j|\le1.
\]
After expansion and collection of equal monomials, the polynomial
has only constant terms, single variables and products of two
distinct variables.
\[
  \widetilde F(x_1,\ldots,x_{12})
  =a_0+\sum_{j=1}^{12}a_ju_j
      +\sum_{1\le j<k\le12}a_{jk}u_ju_k.
\]
Enclose each coefficient by rational endpoints, denoting the lower
endpoint by an underline and the upper endpoint by an overline.
Then
\begin{equation}\label{eq:coefficient-bound}
\begin{aligned}
  \widetilde F(x_1,\ldots,x_{12})\ge{}&
  \underline a_0-
  \sum_{j=1}^{12}\max\{|\underline a_j|,|\overline a_j|\}\\
  &-\sum_{1\le j<k\le12}
       \max\{|\underline a_{jk}|,|\overline a_{jk}|\}.
\end{aligned}
\end{equation}
Since $|u_j|\le1$ and $|u_ju_k|\le1$, each nonconstant term is
bounded below by the negative of its coefficient's maximum possible
magnitude. The estimate holds for every choice of coordinates in
their intervals, hence for every placement in $\mathcal B$.

\paragraph{Checking the geometric conditions.}
The fan and chord conditions reduce to sign tests for determinants
such as $(P_j-P_0)\times(P_k-P_0)$. Each is twice a triangle's
shoelace expression, so the same polynomial bound applies.
We first try direct interval evaluation, cancelling translations
in differences between vertices of the same hull. If this does not
establish the required sign, we use \cref{eq:coefficient-bound}.

Fan conditions require strict positivity. Chord side conditions
allow zero. Distinct endpoints are certified by a coordinate
difference whose interval excludes zero, or by a strict determinant
sign with a third vertex.

\paragraph{Checking the area threshold.}
We apply \cref{eq:coefficient-bound} to $\widetilde F$ to certify
$F\ge\lbsym$. We first combine like terms in $x_j$, substitute
$x_j=m_j+r_ju_j$, and combine like terms in $u_j$.
If this bound is below $\lbsym$, we instead substitute into the
original terms of $\widetilde F$ and combine like terms afterward.
Both orders give valid lower bounds. Rounding can make them differ.

If the geometric checks pass and one of these lower bounds is at
least $\lbsym$, then $\area(H)\ge\lbsym$ throughout $\mathcal B$.

\subsection{The finite certificate}\label{sec:lb-certificate}

\Cref{sec:lb-box} explains how to certify one placement box. We now
apply this method to the reduced domain $\mathcal D$ of the three
angles and six translation coordinates (\cref{eq:domain}).

The computer search starts from a rational box containing $\mathcal D$.
On each box it chooses the following data.
\begin{enumerate}
  \item \emph{Vertex lists.} Each list selects some of the
  $2+3+4+4=13$ labeled vertices of $L,T,U,C$ and fixes their order.
  The search uses triangles and four-point lists, together with
  lists suggested by the convex hull of the approximate midpoint
  placement.
  \item \emph{Rational weights.} A single list is tested with
  weight one. To combine lists, numerical linear programming seeks
  weights that improve an approximation to the coefficient bound
  in \cref{eq:coefficient-bound}. The proposed weights are converted
  to nonnegative rational values summing to one.
\end{enumerate}
The twelve variables in \cref{sec:lb-box} describe the placements
of $T,U,C$ using six translation coordinates and three
cosine--sine pairs. They determine the positions of the vertices,
while $L$ remains fixed.

A box becomes
a \emph{leaf} only when the shoelace bounds of \cref{sec:lb-fans}
and the interval checks of \cref{sec:lb-box} establish
$\lbsym\le F\le\area(H)$ throughout the box.
Otherwise, the search chooses one of the nine coordinates using interval widths and
polynomial coefficients, splits its interval near the midpoint, and
continues on both child boxes. Their union is the parent box.

A \emph{finite certificate} records the completed subdivision tree
and the vertex lists, weights, and any chord decompositions at each
leaf. Acceptance proves the area bound throughout the root box.
Each accepted leaf proves it on its own box, and the two children
of a split cover their parent. Induction on the finite tree
therefore gives $\area(H)\ge\lbsym$ throughout the root and hence
throughout $\mathcal D$.

The supplied certificate is accepted by the Lean verification
described in \cref{sec:verification}. \Cref{app:lb-checks} gives
the checking procedure, and \cref{tab:certificates} records the
certificate statistics.

\begin{proof}[Proof of \cref{thm:lb}]
Suppose a convex set of area less than $\lbsym$ covers $\Sfam$.
Choose a copy of each arc inside it. By convexity, the set contains
their joint hull $H$, so $\area(H)<\lbsym$.
By \cref{sec:lb-domain}, there is a placement in $\mathcal D$ with
the same joint hull area. The leaf boxes cover $\mathcal D$, so this
placement belongs to one of them. Its recorded bounds, justified by
\cref{sec:lb-fans,sec:lb-box}, give
\[
  \lbsym\le F\le\area(H)<\lbsym,
\]
a contradiction. Every convex universal cover covers $\Sfam$, so its
area is at least $\lbsym$. Taking the infimum over these covers gives
$\aopt\ge\lbsym = \lboundvalue$.
\end{proof}

\section{The upper bound}\label{sec:ub}

We construct a quadrilateral $\Kub$ and prove that it covers every unit
arc. An uncovered arc would give an uncovered polygon with a visiting path of length at most one
(\cref{sec:ub-witness}). We express failure to cover this polygon through
support inequalities (\cref{sec:ub-escape}), then relate their weighted
sums to the visiting length (\cref{sec:ub-length}). A finite certificate
combines these inequalities to force a length greater than one
(\cref{sec:ub-certificate}).

For $\sigma\in\mathbb Q$, the rational parametrization
\[
  n(\sigma)=\left(\frac{1-\sigma^2}{1+\sigma^2},
                    \frac{2\sigma}{1+\sigma^2}\right)
\]
gives a unit vector. Define
\begin{equation}\label{eq:K}
  \Kub=\{x\in\R^2:n_i\cdot x\le h_i,\ 0\le i\le3\},
  \qquad n_i=n(\sigma_i),
\end{equation}
where $n_i$ is the outward unit normal and $h_i$ the offset of side $i$.
We use the following rational data.
\[
\begin{array}{c|rrrr}
 i&0&1&2&3\\ \hline
 \sigma_i&-631/58&-1&7/16&37/56\\
 h_i&0&0&130783/250000&446199/1000000
\end{array}
\]
Write $V_1,\ldots,V_4$ for the intersections of consecutive side
lines, starting with lines $0,1$ (\cref{fig:K}).
Side $0$ is $[V_4,V_1]$, and side $i$ is $[V_i,V_{i+1}]$ for $i=1,2,3$.

\begin{theorem}\label{thm:ub}
The quadrilateral $\Kub$ covers every unit arc. Its area is
\[
  \area(\Kub)=\uboundvalue=\ubounddecimal.
\]
\end{theorem}

\begin{figure}[htbp]
  \centering
  \begin{tikzpicture}[scale=5.6]
\fill[black!6] (0,0) -- (0.77080,0) -- (0.45656,0.29039) -- (-0.09765,0.52671) -- cycle;
\draw[->,figaxis] (-0.30,0) -- (0.95,0) node[below,figticklbl] {$x$};
\draw[->,figaxis] (0,-0.04) -- (0,0.62) node[right,figticklbl] {$y$};
\foreach \x in {0.25,0.5}
  \draw[figtick] (\x,0.012) -- (\x,-0.012) node[below,figticklbl] {\x};
\foreach \y in {0.25,0.5}
  \draw[figtick] (-0.012,\y) -- (0.012,\y)
    node[right,figticklbl,fill=white,fill opacity=0.7,text opacity=1,inner sep=0.5pt] {\y};
\draw[thick] (0,0) -- (0.77080,0) -- (0.45656,0.29039) -- (-0.09765,0.52671) -- cycle;
\fill (0,0) circle (0.006) node[below left,yshift=-2pt,figlbl] {$V_1=(0,0)$};
\fill (0.77080,0) circle (0.006) node[below right,figlbl] {$V_2$};
\fill (0.45656,0.29039) circle (0.006) node[above right,figlbl] {$V_3$};
\fill (-0.09765,0.52671) circle (0.006) node[above left,figlbl] {$V_4$};
\draw[->,thick,cbAcc] (0.38540,0) -- (0.38540,-0.13) node[below,figlbl] {$n_1$};
\draw[->,thick,cbAcc] (0.61368,0.14520) -- (0.70191,0.24068) node[right,figlbl] {$n_2$};
\draw[->,thick,cbAcc] (0.17945,0.40855) -- (0.23044,0.52813) node[above,figlbl] {$n_3$};
\draw[->,thick,cbAcc] (-0.04883,0.26335) -- (-0.17665,0.23965) node[left,figlbl] {$n_0$};
\end{tikzpicture}
  \caption{The quadrilateral $\Kub$, drawn to scale.
  Short arrows show the directions of its outward unit normals.}
  \label{fig:K}
\end{figure}
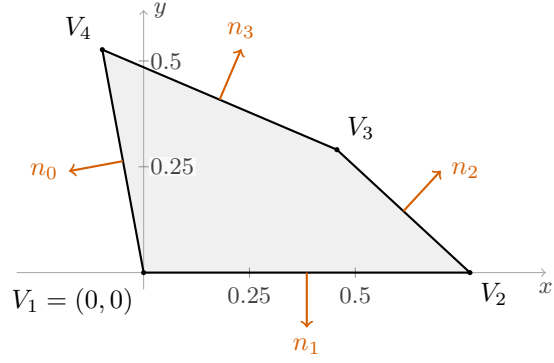

Intersecting the side lines in \cref{eq:K} verifies the quadrilateral,
and the shoelace formula gives its area. The same rational calculation
gives $|V_2-V_4|^2>1$. By convexity, $\Kub$ contains this diagonal and
hence a unit segment. This will rule out collinear witnesses in
\cref{lem:finite}. It remains to prove coverage of all unit arcs.

\subsection{A finite witness}\label{sec:ub-witness}

We first retain finitely many points of a hypothetical uncovered arc
and take their convex hull. We need two orders on the hull's vertices,
the order around the boundary and the order along a shortest visiting path.

A \emph{visiting path} joins all vertices once in some order by straight
segments, without returning to the starting vertex. A \emph{boundary
interval} is a consecutive block in the cyclic boundary order, allowing
wraparound (\cref{fig:witness}(a)). Related boundary-order reductions
appear in \cite[Proposition~3.3]{deng2026}.

\begin{lemma}[finite witness]\label{lem:finite}
If $\Kub$ does not cover a unit arc, it does not cover some convex
polygon $Q$ with $m\ge3$ vertices. Listed in the order of a shortest
visiting path, these vertices $Q_1,\ldots,Q_m$ satisfy
\[
  \ellg_Q=\sum_{j=1}^{m-1}|Q_{j+1}-Q_j|\le1.
\]
Every prefix $\{Q_1,\ldots,Q_j\}$ is a boundary interval.
\end{lemma}

\begin{proof}
Let $\Gamma$ be the trace of an uncovered unit arc. Choose increasing
finite sets $\Gamma_n\subseteq\Gamma$ with dense union, all containing
a fixed point $x_0$. Some $\Gamma_n$ is not covered by $\Kub$.
Otherwise, choose rigid motions $g_n(x)=M_nx+b_n$ placing $\Gamma_n$
in $\Kub$. The translations $b_n$ are bounded because $\Kub$ is
bounded, $g_n(x_0)\in\Kub$, and $|M_nx_0|=|x_0|$.
Compactness of $O(2)$ then gives a subsequence of these motions
converging to a rigid motion $g$. Each point of the union belongs to
all sufficiently large $\Gamma_n$, so its image under $g$ lies in the
closed set $\Kub$. By density and continuity, $g(\Gamma)\subseteq\Kub$,
a contradiction.

Fix an uncovered $\Gamma_n$ and put $Q=\conv(\Gamma_n)$.
Covering $Q$ would also cover $\Gamma_n$. Every vertex of $Q$ belongs
to $\Gamma_n$. Record one visit to each vertex along the arc and join
them in that order. By the definition of arc length, this path has
length at most one, so a shortest visiting path has length
$\ellg_Q\le1$. The hull cannot be collinear, since a collinear hull
would be a segment of length at most one and would fit in $\Kub$.

A shortest visiting path has no crossing edges. Reversing the portion
between two crossing edges and reconnecting the endpoints would
strictly shorten it by the triangle inequality. Strictness holds
because no three vertices of $Q$ are collinear.
The first edge of every suffix of the path must therefore be a boundary
edge of the hull of its remaining vertices. Otherwise it separates
remaining vertices into two nonempty sides. After traversing this edge,
the path cannot revisit either endpoint, so visiting both sides would
force a later edge to cross it.

After the first vertex is removed, the current vertex is an endpoint
of the remaining open boundary list. Its next edge goes either to its
neighbor on that list or to the other endpoint. Removing the current
endpoint preserves this property. Thus the unvisited vertices, and
hence every visited prefix, form a boundary interval.
\end{proof}

Fix such a polygon and path. Write $\tau(Q_j)=j$ for a vertex's
\emph{visiting rank}, and $x(P)$ for the horizontal coordinate of a
point $P$. We next place a bottom edge and a topmost vertex so that
the top vertex is visited between the bottom endpoints. These three
contacts will divide the path into four stages (\cref{fig:witness}(b)).

\begin{lemma}[normal form]\label{lem:normal}
After a rigid motion, $Q$ has a bottom edge $[P_-,P_+]$ on the
$x$-axis and a topmost vertex $P_0$ such that
\[
  x(P_-)<x(P_+),\qquad \tau(P_-)<\tau(P_0)<\tau(P_+).
\]
\end{lemma}

\begin{proof}
Rotate two labeled parallel supporting lines around $Q$, keeping both
in contact with it. Start in a direction where each touches a single
vertex. After a half-turn, the contact vertices exchange lines, so
which line touches the earlier-visited vertex must change.
Such a change occurs at one of the finitely many directions where
a line touches an edge.

At that direction, each line touches a vertex or an edge. Its contacts
immediately before and after are the endpoints of that edge, or the
same vertex when no edge is touched. For each line, take the interval
spanned by these ranks, using a single rank when it touches a vertex. The two rank intervals overlap. Otherwise
the same labeled line would touch the earlier-visited vertex both before
and after the change. Since $Q$ is not collinear, the two touched sets
are disjoint, so their endpoint ranks are distinct. Write the intervals
as $[a,b]$ and $[c,d]$, exchanging the lines so that $a<c$.
Their overlap gives $a<c<b$, so the first line touches an edge.
Let $P_-$ and $P_+$ be its endpoints of ranks $a$ and $b$, and let
$P_0$ be the contact of rank $c$ on the opposite line. Then
\[
  \tau(P_-)=a<c=\tau(P_0)<b=\tau(P_+).
\]
Move the edge onto the $x$-axis with $Q$ above it. Then $P_0$ is
topmost. Reflect in the $y$-axis if necessary to obtain
$x(P_-)<x(P_+)$.
\end{proof}

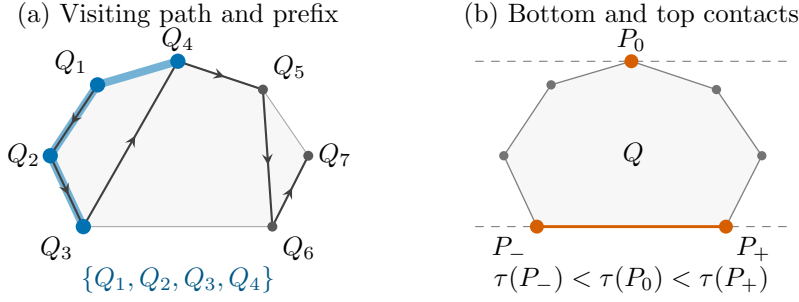
\begin{figure}[t]
  \centering
  \begingroup
\newcommand{\witnessvertices}{%
  \coordinate (a0) at (-1,0);
  \coordinate (a1) at (1,0);
  \coordinate (a2) at (1.38,.75);
  \coordinate (a3) at (.90,1.45);
  \coordinate (a4) at (0,1.75);
  \coordinate (a5) at (-.85,1.50);
  \coordinate (a6) at (-1.35,.75);
}
\begin{tikzpicture}[line join=round,line cap=round]
  \node[figlbl] at (0,2.80) {(a) Visiting path and prefix};
  \node[figlbl] at (6,2.80) {(b) Bottom and top contacts};
  \begin{scope}[scale=1.25]
    \witnessvertices
    \path[fill=black!3,draw=black!35,line width=.4pt]
      (a0)--(a1)--(a2)--(a3)--(a4)--(a5)--(a6)--cycle;
    \draw[cbB!55,line width=3pt] (a4)--(a5)--(a6)--(a0);
    \foreach \i/\j in {5/6,6/0,0/4,4/3,3/1,1/2}{
      \draw[black!75,line width=.8pt,
        postaction={decorate},
        decoration={markings,mark=at position .55 with {\arrow{stealth}}}]
        (a\i)--(a\j);
    }
    \foreach \i in {0,...,6}
      \fill[black!65] (a\i) circle[radius=1.5pt];
    \foreach \i in {5,6,0,4}
      \fill[cbB] (a\i) circle[radius=2.3pt];
    \node[figlbl,above left] at (a5) {$Q_1$};
    \node[figlbl,left] at (a6) {$Q_2$};
    \node[figlbl,below left] at (a0) {$Q_3$};
    \node[figlbl,above] at (a4) {$Q_4$};
    \node[figlbl,above right] at (a3) {$Q_5$};
    \node[figlbl,below right] at (a1) {$Q_6$};
    \node[figlbl,right] at (a2) {$Q_7$};
    \node[figlbl,text=cbB!75!black] at (0,-.56)
      {$\{Q_1,Q_2,Q_3,Q_4\}$};
  \end{scope}
  \begin{scope}[shift={(6,0)},scale=1.25]
    \witnessvertices
    \path[fill=black!3,draw=black!50,line width=.5pt]
      (a0)--(a1)--(a2)--(a3)--(a4)--(a5)--(a6)--cycle;
    \draw[figlead,dashed] (-1.65,0)--(1.65,0);
    \draw[figlead,dashed] (-1.65,1.75)--(1.65,1.75);
    \draw[figseg,cbAcc] (a0)--(a1);
    \foreach \i in {0,...,6}
      \fill[black!55] (a\i) circle[radius=1.4pt];
    \foreach \i in {0,1,4}
      \fill[cbAcc] (a\i) circle[radius=2.1pt];
    \node[figlbl,below left] at (a0) {$P_-$};
    \node[figlbl,below right] at (a1) {$P_+$};
    \node[figlbl,above] at (a4) {$P_0$};
    \node[figlbl] at (.02,.76) {$Q$};
    \node[figlbl] at (0,-.56) {$\tau(P_-)<\tau(P_0)<\tau(P_+)$};
  \end{scope}
\end{tikzpicture}
\endgroup
  \caption{(a) Arrows give the visiting order $Q_1,\ldots,Q_7$.
  The thick boundary edges join $Q_4,Q_1,Q_2,Q_3$ in boundary order,
  showing that the first four visited vertices form a boundary interval.
  (b) The same polygon has
  bottom and top contacts $(P_-,P_0,P_+)=(Q_3,Q_4,Q_6)$.
  The illustrated path is schematic and need not be shortest.}
  \label{fig:witness}
\end{figure}

A fixed bound on the number of segments of a polygonal arc does not
suffice to test universal coverage \cite{pww2007}. We will instead use
a fixed finite list of support contacts in \cref{sec:ub-escape}.

\subsection{Contacts and escape inequalities}\label{sec:ub-escape}

An uncovered polygon must fail to fit in $\Kub$ in every orientation.
We first eliminate translations, then express this failure using
finitely many support contacts. The certificate will use finitely
many orientations to rule out all resulting cases.

\paragraph{Eliminating translations.}
The \emph{support function} of $Q$ is
\[
  h_Q(u)=\max_{x\in Q}u\cdot x,\qquad u\in\R^2.
\]
A vertex attaining this maximum is a \emph{support contact} in direction
$u$. A translate $Q+b$ lies in $\Kub$ exactly when
$h_Q(n_i)+n_i\cdot b\le h_i$ for all four sides.
The normals satisfy
\[
  a_{j0}n_0+a_{j1}n_1+n_j=0,\qquad j\in\{2,3\},
\]
with uniquely determined positive rational coefficients $a_{j0},a_{j1}$.
Their positivity follows from the side data following \cref{eq:K}. Define
\[
  \Psi_j(Q)=a_{j0}h_Q(n_0)+a_{j1}h_Q(n_1)+h_Q(n_j),
  \qquad j\in\{2,3\}.
\]

\begin{lemma}[translation criterion]\label{lem:escape}
A translate of $Q$ lies in $\Kub$ if and only if
$\Psi_2(Q)\le h_2$ and $\Psi_3(Q)\le h_3$.
If $\Kub$ does not cover $Q$, then for every $M\in O(2)$ at least one
of the two inequalities
\begin{equation}\label{eq:escape}
  a_{j0}h_Q(Mn_0)+a_{j1}h_Q(Mn_1)+h_Q(Mn_j)>h_j,
  \qquad j\in\{2,3\},
\end{equation}
holds.
\end{lemma}

\begin{proof}
Since $n_0,n_1$ are independent, choose the unique translation $b$ with
$n_0\cdot b=-h_Q(n_0)$ and $n_1\cdot b=-h_Q(n_1)$.
Then $Q+b$ touches the supporting lines of sides $0,1$, whose offsets
are zero. For the remaining sides, the normal relations give
\[
  h_{Q+b}(n_j)=h_Q(n_j)+n_j\cdot b=\Psi_j(Q),
  \qquad j\in\{2,3\}.
\]
Any further translation $w$ respecting the first two sides satisfies
$n_0\cdot w,n_1\cdot w\le0$, and therefore
\[
  n_j\cdot w=-a_{j0}n_0\cdot w-a_{j1}n_1\cdot w\ge0,
  \qquad j\in\{2,3\}.
\]
Thus $b$ minimizes both remaining support values among translations
respecting the first two sides, proving the criterion. Apply it to
$M^{\mathsf T}Q$ and use $h_{M^{\mathsf T}Q}(n_i)=h_Q(Mn_i)$ to obtain
\cref{eq:escape}. We call these the two \emph{escape alternatives} for $M$.
\end{proof}

\paragraph{Choosing contacts.}
The certificate specifies rational unit directions
$u^R_1,\ldots,u^R_{k_R}$ in the open right half-plane, strictly ordered
from downward to upward, and $u^L_1,\ldots,u^L_{k_L}$ in the open left
half-plane, strictly ordered from upward to downward. Choose support
contacts $R_i,L_j$ in these directions. As the outward direction turns,
support contacts follow the boundary of a convex polygon. Together with
$P_-,P_+,P_0$, they occur in the weak cyclic boundary order
\begin{equation}\label{eq:boundary}
  P_-,\ P_+,\ R_1,\ldots,R_{k_R},\ P_0,\ L_1,\ldots,L_{k_L}.
\end{equation}
Both bottom contacts have direction $(0,-1)$, and the top contact
$P_0$ has direction $(0,1)$. The contacts and their visiting ranks come
from the hypothetical polygon $Q$.

Let $\Lambda$ be the set of these $k_R+k_L+3$ contact labels, and write
$P_\lambda,u_\lambda$ for the point and direction attached to label
$\lambda$. Distinct labels may denote the same polygon vertex.
Sets of contacts below retain these labels separately. The support
contacts and bottom edge satisfy
\begin{equation}\label{eq:support}
  u_\lambda\cdot(P_\lambda-P_\eta)\ge0
  \quad(\lambda,\eta\in\Lambda),\qquad
  x(P_+)-x(P_-)\ge0.
\end{equation}

\paragraph{Linear escape inequalities.}
For each direction $u=Mn_i$ used in \cref{eq:escape}, we compute an
expansion $u=\sum_{\lambda\in\Lambda}\kappa_\lambda u_\lambda$
from the chosen directions, with nonnegative rational coefficients
$\kappa_\lambda$.
For every $x\in Q$, each $u_\lambda\cdot x$ is at most
$h_Q(u_\lambda)$. Taking the maximum over $x$ therefore gives
\[
  h_Q(u)\le\sum_{\lambda\in\Lambda}\kappa_\lambda h_Q(u_\lambda)
    =\sum_{\lambda\in\Lambda}\kappa_\lambda
       u_\lambda\cdot P_\lambda.
\]
A direction already chosen uses a one-term expansion, with $P_+$ used
for the downward direction. Substituting these upper bounds in
\cref{eq:escape} makes its left side no smaller. This gives a necessary
strict linear \emph{escape inequality} in the contact coordinates,
with right side $h_j$. For each chosen $M$, an uncovered $Q$ satisfies
at least one of the two resulting linear escape inequalities.

\subsection{From contact order to a length bound}\label{sec:ub-length}

We want a weighted sum whose left side is at most $\ellg_Q$
and whose right side is at least one. We first show how sums of coefficient
vectors over visited contacts control this left side. We then use the
boundary order to give a finite list of coefficient sums to check.

\paragraph{Why prefix sums control length.}
For $0\le r\le m$, define the \emph{completed prefix}
\[
  \Lambda_r=\{\lambda\in\Lambda:\tau(P_\lambda)\le r\}.
\]
Labels at the same vertex enter together. Consider rational coefficient
vectors $f_\lambda\in\mathbb Q^2$ with
$\sum_{\lambda\in\Lambda}f_\lambda=0$, and set
\[
  w_r=-\sum_{\lambda\in\Lambda_r}f_\lambda,
  \qquad 0\le r\le m.
\]
Then $w_0=w_m=0$. The coefficients attached to vertex $Q_j$ sum to
$w_{j-1}-w_j$, including all contact labels at that vertex.
If $|w_r|\le1$ for every $r$, regrouping by vertex and then by edge gives
\begin{equation}\label{eq:length}
\begin{aligned}
  \sum_{\lambda\in\Lambda}f_\lambda\cdot P_\lambda
   &=\sum_{j=1}^{m}(w_{j-1}-w_j)\cdot Q_j\\
   &=\sum_{j=1}^{m-1}w_j\cdot(Q_{j+1}-Q_j)
     \le\sum_{j=1}^{m-1}|Q_{j+1}-Q_j|=\ellg_Q.
\end{aligned}
\end{equation}
The inequality applies Cauchy--Schwarz to each edge. Related support
and length estimates appear in \cite{temerevDoria2026} and
\cite[Lemma~2.7]{deng2026}.

For $I\subseteq\Lambda$, write $I^c=\Lambda\setminus I$.
Since the total coefficient vector is zero, $I$ and $I^c$ have opposite
coefficient sums, hence equal norms. We may therefore check either
the visited or the unvisited contacts.

\paragraph{Which prefixes can occur.}
By \cref{lem:finite}, every completed contact prefix $\Lambda_r$
is a boundary interval in the order \cref{eq:boundary}.
Any such interval containing $P_-$ and a right contact must contain
$P_+$ or $P_0$. Thus a right contact $R_i$ cannot be visited before
$P_0$. Otherwise the prefix at
$r=\max\{\tau(R_i),\tau(P_-)\}<\tau(P_0)$ would contain $R_i,P_-$
but neither $P_0$ nor $P_+$. Similarly, the prefix at $P_0$ contains
$P_-$ and excludes $P_+$, so it contains every left contact. Hence
\[
  \tau(L_j)\le\tau(P_0)\le\tau(R_i).
\]
The left contacts visited by $P_-$ and the right contacts visited
before $P_+$ form final blocks in their respective boundary lists.
Define integers $0\le p\le k_R$ and $0\le q\le k_L$ by
\begin{equation}\label{eq:contact-pair}
  \{i:\tau(R_i)\ge\tau(P_+)\}=\{1,\ldots,p\},\qquad
  \{j:\tau(L_j)>\tau(P_-)\}=\{1,\ldots,q\}.
\end{equation}
Thus $p$ counts right contacts still unvisited just before $P_+$,
and $q$ counts left contacts still unvisited just after $P_-$.
The different weak and strict inequalities account for coincident contacts.

We check ranges of pairs $(p,q)$ together. Let
\[
  \mathcal R=([p_{\rm lo},p_{\rm hi}]\times[q_{\rm lo},q_{\rm hi}])
       \cap\mathbb Z^2
\]
be a nonempty rectangle with integer endpoints in
$[0,k_R]\times[0,k_L]$.
For fixed $p,q$, there are four stages (\cref{fig:prefix}).
Before $P_-$, the visited contacts form a block in
$L_{q+1},\ldots,L_{k_L}$. They then grow as a final block of the left
list attached to $P_-$. After $P_0$, the unvisited contacts form an
initial block of the right list attached to $P_+$. After $P_+$,
only a block within $R_1,\ldots,R_p$ remains.
In the notation of the table below, these restrictions give
$a\ge q+1$, $j\le q+1$, $s\ge p$, and $b\le p$, respectively.
Allowing $(p,q)$ to range over $\mathcal R$ gives the following family
$\mathcal F(\mathcal R)$ of sets of contact labels.

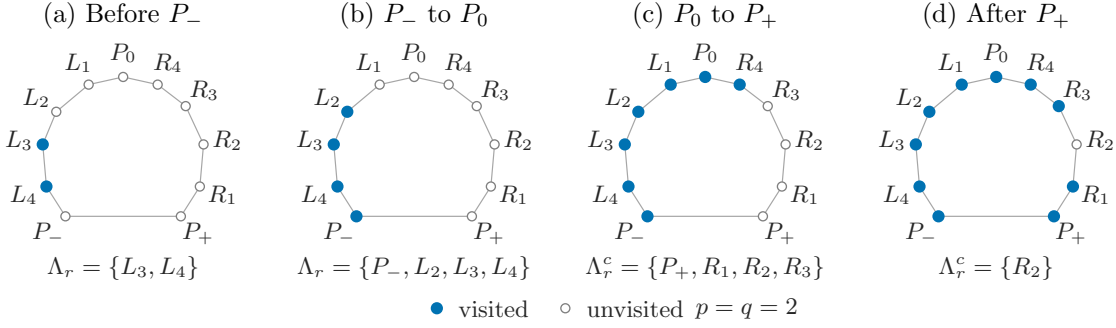
\begin{figure}[t]
  \centering
  \begingroup
\newcommand{\prefixpanel}[4]{%
  \begin{scope}[xshift=#1cm]
    \node[figlbl] at (0,1.90) {#2};
    \foreach \i/\ang/\lab in {
      0/225/P_-,1/315/P_+,2/340/R_1,3/370/R_2,4/400/R_3,
      5/425/R_4,6/450/P_0,7/475/L_1,8/505/L_2,9/530/L_3,10/560/L_4
    }{
      \coordinate (v\i) at (\ang:1.08);
      \node[figsmall] at (\ang:1.40) {$\lab$};
    }
    \draw[black!35,line width=.45pt]
      (v0)--(v1)--(v2)--(v3)--(v4)--(v5)--(v6)
      --(v7)--(v8)--(v9)--(v10)--cycle;
    \foreach \i in {0,...,10}
      \draw[draw=black!50,fill=white,line width=.5pt]
        (v\i) circle[radius=1.7pt];
    \foreach \i in {#3}
      \fill[cbB] (v\i) circle[radius=2.3pt];
    \node[figsmall,align=center] at (0,-1.45) {#4};
  \end{scope}%
}
\begin{tikzpicture}
  \prefixpanel{0}{(a) Before $P_-$}
    {9,10}{$\Lambda_r=\{L_3,L_4\}$}
  \prefixpanel{3.85}{(b) $P_-$ to $P_0$}
    {0,8,9,10}{$\Lambda_r=\{P_-,L_2,L_3,L_4\}$}
  \prefixpanel{7.70}{(c) $P_0$ to $P_+$}
    {0,5,6,7,8,9,10}{$\Lambda_r^c=\{P_+,R_1,R_2,R_3\}$}
  \prefixpanel{11.55}{(d) After $P_+$}
    {0,1,2,4,5,6,7,8,9,10}{$\Lambda_r^c=\{R_2\}$}
  \fill[cbB] (4.15,-1.98) circle[radius=2.3pt];
  \node[figsmall,anchor=west] at (4.30,-1.98) {visited};
  \draw[draw=black!50,fill=white,line width=.5pt]
    (5.85,-1.98) circle[radius=1.7pt];
  \node[figsmall,anchor=west] at (6.00,-1.98) {unvisited};
  \node[figsmall] at (8.25,-1.98) {$p=q=2$};
\end{tikzpicture}
\endgroup
  \caption{Four stages with $k_R=k_L=4$ and $p=q=2$.
  $\Lambda_r$ is the set of contact labels visited by rank $r$
  (filled dots), and $\Lambda_r^c$ is the set of remaining contact
  labels (open circles).}
  \label{fig:prefix}
\end{figure}

\begin{center}\small
\setlength{\tabcolsep}{5pt}
\begin{tabular}{lll}
\toprule
Stage & Prefix or complement & Range\\
\midrule
$r<\tau(P_-)$
 & $\Lambda_r=\{L_a,\ldots,L_b\}$ & $q_{\rm lo}+1\le a\le b\le k_L$\\
$\tau(P_-)\le r<\tau(P_0)$
 & $\Lambda_r=\{P_-,L_j,\ldots,L_{k_L}\}$ & $1\le j\le q_{\rm hi}+1$\\
$\tau(P_0)\le r<\tau(P_+)$
 & $\Lambda_r^c=\{P_+,R_1,\ldots,R_s\}$ & $p_{\rm lo}\le s\le k_R$\\
$\tau(P_+)\le r$
 & $\Lambda_r^c=\{R_{a+1},\ldots,R_b\}$ & $0\le a\le b\le p_{\rm hi}$\\
\bottomrule
\end{tabular}
\end{center}

Empty index ranges are allowed. For a witness with $(p,q)\in\mathcal R$,
every nonempty proper completed prefix satisfies
$\Lambda_r\in\mathcal F(\mathcal R)$ or
$\Lambda_r^c\in\mathcal F(\mathcal R)$.
Empty and full prefixes give $w_r=0$. Vertices with no contact label
do not change a prefix sum. Thus $O(k_R^2+k_L^2)$ coefficient checks
cover every visiting rank, independently of the number of vertices of $Q$.

\subsection{The finite certificate}\label{sec:ub-certificate}

A computer search builds a finite tree of cases. Each leaf records
rational multipliers whose weighted inequality, combined with
\cref{eq:length}, forces $\ellg_Q>1$.

\paragraph{Covering all cases.}
The root has no escape assumptions and the full integer rectangle
$([0,k_R]\times[0,k_L])\cap\mathbb Z^2$ of possible pairs $(p,q)$.
An internal node either branches into both escape alternatives for
one specified orthogonal matrix $M$, or covers its rectangle by child
rectangles. Escape assumptions are inherited down the tree.
Each leaf uses only the inherited escape inequalities.
By \cref{lem:escape} and rectangle coverage, every uncovered polygon
reaches a leaf whose assumptions it satisfies.

\paragraph{Certifying a leaf.}
A leaf handles one rectangle $\mathcal R$ and the inherited escape choices.
It supplies nonnegative rational multipliers $\mu_k$ on the support and
bottom-edge inequalities in \cref{eq:support} and the selected linear
escape inequalities. Let $\mathcal E_2,\mathcal E_3$ index the selected
escape inequalities with right sides $h_2,h_3$.
Collecting the coefficients gives rational vectors $f_\lambda$ and
\[
  \sum_{\lambda\in\Lambda}f_\lambda\cdot P_\lambda\ge B,
  \qquad
  B=h_2\sum_{k\in\mathcal E_2}\mu_k
      +h_3\sum_{k\in\mathcal E_3}\mu_k.
\]
The support and bottom-edge inequalities contribute zero to the right side.
We reconstruct $f_\lambda$ from the multipliers and check
\begin{equation}\label{eq:checks}
  B\ge1,\qquad
  \sum_{\lambda\in\Lambda}f_\lambda=0,\qquad
  \left|\sum_{\lambda\in I}f_\lambda\right|^2\le1
       \quad(I\in\mathcal F(\mathcal R)).
\end{equation}
The squared norms keep these comparisons rational.
Because $B\ge1$ and $h_2,h_3>0$, some escape inequality has a positive
multiplier. Escape inequalities are strict, so the weighted sum is
strictly greater than $B$. When $(p,q)\in\mathcal R$, the prefix checks
and \cref{eq:length} bound the same sum by $\ellg_Q$.

All checks use the fixed side data following \cref{eq:K} and target length $1$.
They verify the direction conditions, orthogonality, the nonnegative
expansions, multiplier nonnegativity, and branch coverage using exact
rational arithmetic. The supplied certificate passes these checks.
Its formal verification, size and run times are described in
\cref{sec:verification,tab:certificates}.

\begin{proof}[Proof of \cref{thm:ub}]
Suppose $\Kub$ does not cover a unit arc. By \cref{lem:finite,lem:normal},
there is an uncovered polygon $Q$ in normal form with
$\ellg_Q\le1$. Choose its contacts as in \cref{sec:ub-escape} and
its pair $(p,q)$ from \cref{eq:contact-pair}.
Follow the certificate tree through a true escape alternative at each
orientation branch and a rectangle containing $(p,q)$ at each subdivision.
At the resulting leaf, all selected inequalities hold and the prefix
checks apply. Hence
\[
  1\le B<\sum_{\lambda\in\Lambda}f_\lambda\cdot P_\lambda
        \le\ellg_Q\le1,
\]
a contradiction. Thus $\Kub$ covers every unit arc, and its rational
area calculation completes the proof.
\end{proof}

\begin{proof}[Proof of \cref{thm:main}]
\Cref{thm:lb} gives $\aopt\ge\lboundvalue$.
By \cref{thm:ub}, $\Kub$ is a convex universal cover whose area is the
stated rational upper bound. Thus $\aopt\le\area(\Kub) = \uboundvalue = \uboundshort$.
\end{proof}

\section{Certificates and formalization}\label{sec:verification}

Both bounds in \cref{thm:main} have been verified by Lean's kernel.
We explain how the certificates of
\cref{sec:lb-certificate,sec:ub-certificate} complete the formal proof,
report the final axiom audit, and record the data needed to reproduce
the verification.

\paragraph{Kernel verification.}
The geometric arguments and certificate checkers are formalized in
Lean~4 \cite{lean4} with mathlib \cite{mathlib}. The search programs
may use floating-point arithmetic to produce rational certificates,
which are exported to Lean declarations. The finite numerical checks
are proved using \lean{decide +kernel}. Formal soundness theorems
show that acceptance implies the geometric bounds. The kernel checks
the acceptance proofs, the soundness proofs and their final combination
(\cref{fig:verification-flow}). Correctness of the search programs is
not assumed, and no \lean{native\_decide} is used.

\begin{figure}[htbp]
  \centering
  \begin{tikzpicture}[
  x=1cm,y=1cm,>=stealth,
  proofstage/.style={
    draw=black!45,fill=white,line width=.4pt,rounded corners=2pt,
    font=\small,align=center,text width=2.5cm,minimum height=.95cm,
    inner sep=3pt
  },
  flowarrow/.style={->,draw=black!65,line width=.5pt}
]
  \path[fill=cbB!4,draw=cbB!50,line width=.4pt,rounded corners=3pt]
    (5.05,-1.95) rectangle (11.95,1.07);
  \node[figsmall,text=black!85] at (8.5,.79) {Checked by the Lean kernel};

  \node[proofstage,text width=1.9cm] (search) at (0,0) {Search};
  \node[proofstage] (certificate) at (3,0) {Rational\\certificates};
  \node[proofstage] (acceptance) at (6.7,0) {Acceptance\\proofs};
  \node[proofstage] (soundness) at (6.7,-1.25) {Soundness\\proofs};
  \node[proofstage] (bounds) at (10.25,-.625) {Bounds\\\cref{thm:main}};

  \draw[flowarrow] (search.east)--(certificate.west);
  \draw[flowarrow] (certificate.east)--(acceptance.west);
  \draw[flowarrow] (acceptance.east)--(bounds.north west);
  \draw[flowarrow] (soundness.east)--(bounds.south west);
\end{tikzpicture}
  \caption{From search to the bounds. Certificate acceptance and geometric
  soundness provide the two proof inputs to the final theorem.
  Every proof in the shaded region is checked by Lean's kernel.}
  \label{fig:verification-flow}
\end{figure}
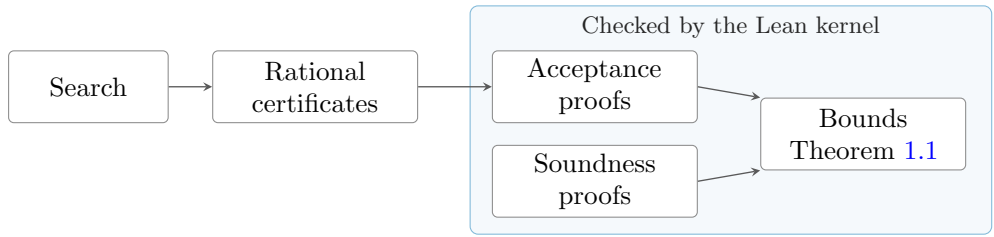

\paragraph{The final theorem and its axioms.}
The combined theorem \lean{MoserWorm.bounds} states \cref{thm:main}.
After importing \lean{MoserWorm}, the command
\lean{\#print axioms MoserWorm.bounds} reports
{\small
\begin{verbatim}
'MoserWorm.bounds' depends on axioms: [propext, Classical.choice, Quot.sound]
\end{verbatim}
}
These are propositional extensionality, classical choice and quotient
soundness. The build audits the theorem's dependencies and rejects
every other axiom, including \lean{sorryAx} for admitted proofs.

\paragraph{Reproduction and timings.}
Readers can verify the supplied compressed certificates without running
the search, or regenerate them and run the same checks. The repository provides separate
commands for each bound and pins Lean 4.33.1 and a fixed mathlib version.
The source code, certificates and
instructions are available at \GithubURL.

\Cref{tab:certificates} records the certificate sizes and measured wall
times on Linux with an AMD EPYC 9R14 processor, 64 logical CPUs and
128\,GiB RAM. Generation includes search, pruning and compression.
Verification includes proof export, kernel checking, the final theorem
build and the axiom audit.

\begin{table}[!htbp]
  \centering
  \small
  \begin{tabular}{lrr}
    \toprule
    & Lower & Upper\\
    \midrule
    \multicolumn{3}{l}{\emph{Trees and files}}\\
    Total tree nodes & $1\,052\,327$ & $24\,898$\\
    Leaves & $526\,164$ & $12\,663$\\
    Maximum tree depth & $68$ & $26$\\
    Compressed data (bytes) & $9\,074\,308$ & $22\,877\,184$\\
    Uncompressed data (bytes) & $19\,676\,535$ & $103\,447\,953$\\
    \midrule
    \multicolumn{3}{l}{\emph{Lean and measured wall times}}\\
    Generated Lean modules & $24\,287$ & $1\,636$\\
    Certificate generation & 1 h 24 min & 2 min \\
    Verification and axiom audit & 3 h 56 min & 59 min\\
    Combined theorem and audit & \multicolumn{2}{c}{27 s}\\
    Total verification, both bounds & \multicolumn{2}{c}{4 h 55 min}\\
    Total including generation, both bounds & \multicolumn{2}{c}{6 h 21 min}\\
    \bottomrule
  \end{tabular}
  \caption{Statistics for the supplied certificates. Lower node counts
  and depth include the ten initial subdivision levels joining $1024$
  stored subtrees. The root has depth zero.}
  \label{tab:certificates}
\end{table}

\section{Conclusion}\label{sec:conclusion}

We proved the new bounds $\lboundvalue\le\aopt\le\uboundshort$
(\cref{thm:main}), reducing the gap between the previous refereed
bounds by approximately $75.2\%$.
In \cref{sec:lb}, finite subdivision bounds the joint convex hull
area of the four arcs in \cref{fig:family} in every placement.
In \cref{sec:ub}, support and length inequalities show that any arc
not covered by the quadrilateral in \cref{fig:K} must have length
greater than one. Both proofs are formalized in Lean~4 and verified
by the Lean kernel (\cref{sec:verification}).

\paragraph{Further directions.}
The numerical placement in \cref{fig:joint-placement} suggests that
the chosen four arcs leave little room to raise the lower bound.
Larger gains may need different or additional arcs, along with
stronger geometric reductions to keep the larger placement search
manageable. For the upper bound, our quadrilateral has not been
proved optimal, and the length estimates may be too weak to certify
smaller covers. We would be interested to know whether the support
and length method of \cref{sec:ub} extends to polygons with more
sides and yields smaller universal covers. The arcs that are
hardest to fit might also offer clues to the shape of an optimal
cover. Determining the optimum would require
a convex universal cover and a matching lower bound.

\section*{Acknowledgments}

The author thanks Hendrik Van Maldeghem for his helpful comments on an
early version of the lower-bound argument.

This work relied heavily on generative AI tools. The mathematical content in
this paper is the product of AI-assisted research steered by the
author. The author accessed models from Anthropic (Claude Opus and
Fable families) and OpenAI (GPT-5 family and Astra)
through the agentic coding tools Claude Code and Codex as well as
through web chat interfaces. These tools carried out the exploration,
the discovery of the geometric reductions, the development and
formalization of the proofs in Lean 4, the implementation of the code, the
computational experiments, and the drafting of the paper.
The author reviewed and verified all AI-generated output and takes
full responsibility for the mathematical claims, code, and exposition.

\appendix
\section{The five-point area bound}\label{app:five-point}

We prove the $m=5$ case of \cref{lem:small-polygon}. The signed shoelace
expression of any ordered five points is at most the area of their convex hull.

\begin{proof}
Let $H_0$ be the original hull of the five points.
If $H_0$ is contained in a line, every signed area is zero.
Otherwise, the shoelace expression is affine in each point separately.
We may therefore move each point in turn to a vertex of $H_0$ without
decreasing the expression.

First suppose three consecutive points $p,q,r$ satisfy
$(q-p)\times(r-q)\le0$.
Deleting $q$ changes the signed area by
$-\tfrac12(q-p)\times(r-q)\ge0$.
The four-point bound then proves the claim.
A repeated point among a cyclic list of five creates either a zero
edge or a triple with equal first and last points, so it gives this case.
It remains to consider five distinct hull vertices with positive turns.

Number the vertices of $H_0$ counterclockwise. Let
$d_i\in\{1,2,3,4\}$ be the counterclockwise step between successive
vertices in the chosen list, with indices read modulo five.
Three distinct vertices of $H_0$ have positive orientation exactly when
they occur in counterclockwise cyclic order. Thus the two steps in a
positive-turn triple traverse less than a full circuit, so
$d_i+d_{i+1}<5$, or equivalently $d_i+d_{i+1}\le4$. Hence
\[
  5\le\sum_i d_i,\qquad
  2\sum_i d_i=\sum_i(d_i+d_{i+1})\le20.
\]
Returning to the starting vertex makes $\sum_i d_i$ a multiple of
five, so it is either five or ten.
If it is five, every step equals one. The list follows the boundary
and its signed area is exactly $\area(H_0)$.
If it is ten, every adjacent sum equals four.
The cycle has odd length, so every step equals two.
This is the pentagram order.

For the remaining case, label the hull vertices $a,b,c,d,e$
counterclockwise. An orientation-preserving affine change of coordinates
puts
\[
 a=(0,0),\quad b=(1,0),\quad e=(0,1),\quad
 c=(x,y),\quad d=(z,w),
\]
where $y,z\ge0$ by convexity.
The change multiplies all signed areas and the hull area by the same
positive factor.

Twice the signed area of the pentagram $(a,c,e,b,d)$ is $x+w-1$.
We compare it with three quadrilaterals in $H_0$.
\[
\begin{array}{c|ccc}
 \text{list}&(a,b,d,e)&(a,b,c,e)&(a,c,d,e)\\ \hline
 \text{twice the signed area}&w+z&x+y&xw-yz+z
\end{array}
\]
If either of the first two expressions is at least $x+w-1$, the
four-point bound finishes the proof.
Otherwise $x-z>1$ and $w-y>1$, and
\[
 (xw-yz+z)-(x+w-1)
 =(x-z-1)(w-y-1)+y(x-z-1)+z(w-y)\ge0.
\]
The third quadrilateral then bounds the pentagram.
Thus the five-point expression is at most $\area(H_0)$.
\end{proof}

\section{Lower-bound certificate checks}\label{app:lb-checks}

The following details supplement \cref{sec:lb-certificate}.

\paragraph{Choosing the leaf data.}
The main choices in the search are as follows.
\begin{enumerate}
  \item \emph{Vertex lists.} Simple candidates include triangles
  using both endpoints of $L$ and a third vertex, and four-point
  lists using two vertices from one hull and two from another,
  ordered alternately.
  Further candidates are obtained from the convex hull of the
  approximate midpoint placement by deleting vertices, or by
  inserting or replacing a vertex with one outside that list.
  Candidates are compared using their area bounds and the
  required geometric checks.
  \item \emph{Weights.} For each candidate list, the search computes
  the centered polynomial coefficients of \cref{sec:lb-box}.
  Numerical linear programming seeks weights that maximize an
  approximation to the lower bound in \cref{eq:coefficient-bound}.
  Coefficients of opposite signs can cancel when the lists are
  combined, reducing the amounts subtracted from the constant
  coefficient. The proposed weights are rounded and adjusted to
  be nonnegative and sum to one. The combined polynomial is then
  rebuilt using these rational weights before the interval checks.
  \item \emph{Reusing the data.} Lists and weights proposed for a
  parent box can also be tried on its children. After subdivision,
  later passes attempt to replace groups of leaves by a single
  leaf on their parent box, reusing lists, choosing new weights or
  trying chord decompositions. Every replacement must pass the
  same geometric and area checks throughout the parent box.
\end{enumerate}
These choices guide the search. Acceptance always requires the
checks below on the whole box.

\paragraph{Subdivision and leaf data.}
The root is formed from the box in \cref{eq:domain} by replacing
$X$ with $0.603894$ and enclosing all endpoints using the outward
interval arithmetic below. Since $X<0.603894$ and $Y=0.478$,
the root contains $\mathcal D$.
Every lower interval endpoint is at most its upper endpoint.
A split uses the midpoint of the selected interval, rounded down
to the $2^{-56}$ grid. Both children include this point and retain
every other coordinate interval.
Each vertex list has at least three entries, all indices valid and
pairwise distinct. Different indices may still denote coincident
points. The weights satisfy $\lambda_\nu\ge0$ and
$\sum_\nu\lambda_\nu\le1$.

\paragraph{Numerical enclosures.}
Numerical interval endpoints are integer multiples of $2^{-56}$.
Lower endpoints are rounded down and upper endpoints up.
Integer square-root bounds enclose the algebraic constants, and
argument-reduced Taylor bounds with remainder estimates enclose
sine and cosine. The bounds for $\pi$ are justified by mathlib.

At each leaf, angle centers lie on this grid and satisfy $|\phi_i^0|\le64$, the
range covered by the trigonometric routine. Each angle interval must
lie in $[\phi_i^0-\rho_i,\phi_i^0+\rho_i]$ with
$0\le\rho_i<\pi/2$, as required by \cref{lem:rotation-reduction}.

\paragraph{Acceptance.}
The geometric conditions in \cref{sec:lb-fans} must hold throughout
each leaf box. Fan inequalities are strict, chord side inequalities
allow zero, and chord endpoints must be distinct. The lower estimate
in \cref{eq:coefficient-bound} must be at least $\lbsym$.
The interval operations and their soundness proofs are formalized
in Lean (\cref{sec:verification}).

\Cref{alg:lb-check} applies these checks recursively, starting with
the root box and the complete certificate tree. Any failed check
returns rejection immediately.

\begin{algorithm}[H]
  \captionsetup{labelsep=period}
  \caption{Checking a certificate tree}
  \label{alg:lb-check}
  \begin{algorithmic}
    \STATE \textbf{Input} a box $\mathcal B$ and a certificate node
    \IF{the node is a split}
      \STATE Form both child boxes in the recorded coordinate
      \STATE Apply this procedure to each child box and its child node
      \RETURN whether both calls accept
    \ENDIF
    \STATE Check the leaf data and angle conditions
    \STATE Construct the interval enclosures
    \STATE Check the required fan and chord conditions
    \RETURN whether \cref{eq:coefficient-bound} gives $F\ge\lbsym$
  \end{algorithmic}
\end{algorithm}

\bibliographystyle{plain}
\bibliography{refs}
\end{document}